%% file: ph2.tex
\documentclass[hidelinks, final]{siamart190516}
\newcommand{\linelabel}[1]{}
\usepackage{amsmath,amssymb}
\usepackage{smalljeff}
\usepackage{mathtools}
\usepackage{tikz,pgfplots,pgfplotstable}
\pgfplotsset{compat=newest}
\usepgfplotslibrary{groupplots}
\usetikzlibrary{calc}

\newcommand{\Htwo}{{\mathcal{H}_2}}

\newcommand{\iu}{\mathrm{i}\mkern1mu}

\newcommand{\PHtwo}{{PH2}}

\newcommand{\TheTitle}{$\Htwo$-Optimal Model Reduction \\ Using Projected Nonlinear Least Squares}
\newcommand{\TheAuthors}{Jeffrey M. Hokanson and Caleb C. Magruder}
\headers{$\Htwo$-Optimal Model Reduction}{\TheAuthors}

\title{{\TheTitle}\thanks{Submitted to the editors 15 June 2018.
\funding{The first author's work is partially supported by DARPA's program Enabling Quantification of Uncertainty in Physical Systems
(EQUiPS). 
}}}

\author{Jeffrey M. Hokanson\thanks{
	Department of Computer Science, 
	University of Colorado Boulder,
	1111 Engineering Dr, Boulder, CO 80309,
	(\email{Jeffrey.Hokanson@colorado.edu}).}
	\and Caleb C. Magruder\thanks{
	The MathWorks Inc.,
	1 Apple Hill Drive, Natick, MA, 01760-2098
	(\email{ccmagruder@gmail.com})}
}

\begin{document}

\maketitle

\input{abstract}
\input{metadata}
\input{intro}

\input{properties}

\input{existing}

\input{projection}

\input{outer}

\input{ratfit}

\input{results}

\input{discussion}
\appendix
\input{appendix}
\input{acknowledgements}
\bibliographystyle{siamplain}
\bibliography{ph2, master, abbrevjournals}

\end{document}

%% file: abstract.tex
\begin{abstract}
In many applications throughout science and engineering,
model reduction plays an important role 
replacing expensive large-scale linear dynamical systems 
by inexpensive reduced order models that capture key features of the original, full order model. 
One approach to model reduction finds reduced order models that are locally optimal approximations 
in the $\Htwo$ norm,
an approach taken by the Iterative Rational Krylov Algorithm (IRKA) among others.
Here we introduce a new approach for $\Htwo$-optimal model reduction using 
the projected nonlinear least squares framework previously introduced in 
[J.\ M.\ Hokanson, SIAM J.\ Sci.\ Comput.\ 39 (2017), pp.~A3107--A3128].
At each iteration, we project the $\Htwo$ optimization problem onto a finite-dimensional subspace
yielding a weighted least squares rational approximation problem.
Subsequent iterations append this subspace
such that the least squares rational approximant 
asymptotically satisfies the first order necessary conditions of the original, $\Htwo$ optimization problem.
This enables us to build reduced order models
with similar error in the $\Htwo$ norm 
but using far fewer evaluations of the expensive, full order model
compared to competing methods.
Moreover, our new algorithm only requires access to the transfer function of the full order model,
unlike IRKA which requires a state-space representation or 
TF-IRKA which requires both the transfer function and its derivative. 
Applying the projected nonlinear least squares framework to the $\Htwo$-optimal model reduction problem 
open new avenues for related model reduction problems.
\end{abstract}

%% file: metadata.tex
\begin{DOI}

\end{DOI}

%% file: intro.tex
\section{Introduction}
Model reduction replaces an expensive, high-fidelity
simulation with a low-dimensional, inexpensive surrogate.
Although the cost of building this \emph{reduced order model} (ROM)
is often comparable to simulating of the original \emph{full order model} (FOM),
this cost is justified in two settings:
many-query settings, such as optimization and uncertainty quantification,
and in real-time applications where the full order simulation is unaffordable.
There are a wide variety of approaches for model reduction,
such as Balanced Truncation, Proper Orthogonal Decomposition (POD), 
and interpolatory methods;
for an extensive overview, see~\cite{BCOW17}.
In this paper, we address the $\Htwo$-optimal model reduction problem---%
a model reduction approach requiring both the full and reduced order models
to be described by a stable, linear, time-invariant dynamical system.
In this setting, 
the action of these models are completely defined through either
their \emph{impulse response} $h$
or their \emph{transfer function} $H$, the Laplace transform of $h$.
Under these assumptions, 
the transfer function $H$ is an element of the $\Htwo$ Hilbert space:
the space of functions analytic in the closed right half plane along with the inner product
\begin{equation}\label{eq:H2ip}
	\langle F, G \rangle_\Htwo \coloneqq \frac{1}{2\pi} \int_{-\infty}^\infty \conj{F(\iu \omega)} G(\iu \omega) \D w,
	\quad F, G\in \Htwo,
	\quad \iu \coloneqq \sqrt{-1}.
\end{equation}
Given a space of candidate reduced order models 
whose transfer functions are in the set $\set S\subset \Htwo$,
the goal of $\Htwo$-optimal model reduction
is to find the element $H_r\in \set S$
that is closest to the transfer function of the full order model $H$
in the $\Htwo$-norm
\begin{equation}\label{eq:H2MOR_gen}
	\min_{H_r\in \set S} \|H  - H_r\|_\Htwo^2
	\quad 
	\text{where} \quad \|F\|_\Htwo^2 \coloneqq \langle F, F\rangle_\Htwo.
\end{equation}
As we consider a case where $\set S$ is nonconvex,
we consider local optimizers satisfying the first order optimality conditions
as solutions to~\cref{eq:H2MOR_gen}.

Here we employ the \emph{projected nonlinear least squares framework} introduced 
by Hokanson~\cite{Hok17} to solve the $\Htwo$-optimal model reduction problem~\cref{eq:H2MOR_gen}.
This framework replaces~\cref{eq:H2MOR_gen}
by a sequence of small, finite-dimensional \emph{projected problems}
whose optimizers converge to an optimizer of \cref{eq:H2MOR_gen}.
Each projected problem introduces an orthogonal projector $P(\ve\mu)$
into the norm, solving
\begin{equation}
	\label{eq:H2MOR_gen_proj}
		\min_{H_r \in \set S}
			\|P(\ve\mu)[H - H_r]\|_\Htwo^2.
\end{equation}
Exploiting the \emph{reproducing kernel} structure of $\Htwo$ 
(see discussion in \cref{sec:properties:RKHS}),
we construct these projectors
using kernel vectors $v[\mu] \in \Htwo$ where $\langle v[\mu], F\rangle_\Htwo = F(\mu)$ 
such that the range of $P(\ve\mu)$ is spanned by $\lbrace v[\mu_j]\rbrace_{j=1}^n$.
This allows the objective of~\cref{eq:H2MOR_gen_proj}
to be rewritten as a weighted least squares problem
involving the mismatch $H - H_r$ evaluated at the \emph{interpolation points} $\ve\mu = [\mu_1,\ldots,\mu_n]$ 
\begin{equation}\label{eq:projected_equiv}
	\|P(\ve\mu)(H - H_r)\|_\Htwo^2 
	= \left\|
		\ma M(\ve\mu)^{-\frac12}
		\left(
		\begin{bmatrix}
			H(\mu_1) \\ \vdots \\ H(\mu_n)
		\end{bmatrix}
		-
		\begin{bmatrix}
			H_r(\mu_1) \\ \vdots \\ H_r(\mu_n)
		\end{bmatrix}
		\right)
	\right\|_2^2,
\end{equation}
where $\ma M(\ve\mu)$ is a Gram matrix of  
the kernel vectors $v[\mu_j]$ acting to preserve the norm;
i.e., $[\ma M(\ve\mu)]_{j,k} = \langle v[\mu_j],v[\mu_k]\rangle_\Htwo$
with explicit definition in~\cref{eq:Mdef}.
As the projected problem is a finite dimensional nonlinear least squares problem,
we can find local minimizers of~\cref{eq:H2MOR_gen_proj}
using standard optimization techniques given a parameterization of $\set S$;
we refer to this as \emph{inner loop}.

Around this inner loop 
we construct an \emph{outer loop}
updating the projector $P(\ve\mu)$
such that local optimizers of the projected problem~\cref{eq:H2MOR_gen_proj} converge 
to a local optimizer of the original problem~\cref{eq:H2MOR_gen}.
The following lemma provides the key insight 
connecting local optimizers of these two problems.
\begin{lemma}\label{lem:projection}
	Suppose $\set S\subset \Htwo$ admits a parameterization by $H_r(\cdot; \ve \theta)$
	such that $\set S = \lbrace H_r(\cdot;\ve\theta) : \ve \theta \in \set D\rbrace$
	and $\set D\subset \R^p$ is open. 
	Let $\set T(\ve\theta)$ denote subspace containing the derivatives of $H_r$ with
	respect to this parameterization at $\ve\theta$,
	\begin{equation}
	\set T(\ve\theta) \coloneqq
	\Span 
	\left\lbrace 
		\tfrac{\partial H_r(\cdot, \ve\theta)}{\partial \theta_1}, 
		\ldots, 
		\tfrac{\partial H_r(\cdot, \ve\theta)}{\partial \theta_p} 
	\right\rbrace.
	\end{equation}
	If $H_r(\cdot, \hve \theta)$ is a local optimizer of the projected problem~\cref{eq:H2MOR_gen_proj}
	and $\Range P(\ve\mu) \supseteq \set T(\hve\theta)$,
	then $H_r(\cdot, \hve \theta)$ is also a local optimizer of the original problem~\cref{eq:H2MOR_gen}.
\end{lemma}
\begin{proof}
If $H_r(\cdot, \hve \theta) \in \set M$ is a local optimizer
of the projected~\cref{eq:H2MOR_gen_proj} problem,
then  
\begin{align}
	\label{eq:opt_proj}
	\langle P(\ve\mu)[H - H_r(\cdot, \hve \theta)], T  \rangle_\Htwo &= 0
	\quad \forall T \in \set T(\hve \theta).
\end{align}
As $P(\ve\mu)$ is an orthogonal projector
we can move it into the right argument
and noting $P(\ve\mu) T= T$
since, by assumption, $\set T(\hve \theta) \subseteq \Range P(\ve\mu)$,
we have 
\begin{equation}
	\langle P(\ve\mu)[H - H_r(\cdot, \hve \theta)], T  \rangle_\Htwo 
	= \langle H - H_r(\cdot, \hve \theta), P(\ve\mu) T  \rangle_\Htwo 
	=\langle H - H_r(\cdot, \hve \theta), T  \rangle_\Htwo. 
\end{equation}
Thus $H(\cdot, \hve \theta)$ is a local optimizer of the 
original problem~\cref{eq:opt_orig} since
\begin{align}
	\label{eq:opt_orig}
	\langle H - H_r(\cdot, \hve \theta), T \rangle_\Htwo &= 0
	\quad \forall T \in \set T(\hve \theta).
\end{align}
\end{proof}
Motivated by this insight,
we design the outer loop to increase the range of the projector $P(\ve\mu)$
such that (asymptotically) we satisfy the containment requirement of \cref{lem:projection}. 
To maximize reuse existing evaluations of the transfer function,
each iteration of the outer loop adds an interpolation point to the projector; i.e., 
\begin{equation}
	\ve\mu^+ = 
	[\mu_1,\ldots,\mu_n,\mu_{n+1}]
	\quad \Longrightarrow \quad
	\Range P(\ve\mu) \subset \Range P(\ve\mu^+).
\end{equation}

\subsection{Space of ROMs\label{sec:intro:rom}}
Although the projected nonlinear least squares framework can be applied to any family of 
reduced order models $\set S$ with a finite-dimensional parameterization,
the details of both the inner and outer loop depend on this choice.
Due to their popularity,
here we restrict our attention to state-space reduced order models
posed over the field $\mathbb{F}$ that is either real $\R$ or complex $\C$:
\begin{gather}
\left\lbrace
\begin{aligned}
	\ve x_r'(t) &= \ma A_r \ve x_r(t) + \ve b_r u(t), \quad \ve x_r(0) = \ve 0 \\
	y_r(t) &= \ve c_r^* \ve x_r(t)
\end{aligned}
\right\rbrace,
\quad\!\!
\text{where}
\quad \!\!
\ma A_r \in \mathbb{F}^{r \times r}, \  \ve b_r, \ve c_r \in \mathbb{F}^r,
\end{gather}
and whose transfer function is
\begin{equation}\label{eq:ROM_tf}
	H_r(z; \ma A_r, \ve b_r, \ve c_r) \coloneqq \ve c_r^*[z \ma I - \ma A_r]^{-1}\ve b_r.
\end{equation}
This transfer function is a degree $(r-1,r)$ rational function 
as the resolvent $[z\ma I - \ma A_r]^{-1}$
is a degree $(r-1,r)$ matrix-valued rational function~\cite[Chap.~1, eq.~(5.23)]{Kat95}.
We denote the space of $(r-1,r)$ scalar-valued rational functions as
\begin{equation}
		\set R_r (\mathbb{F})
		\coloneqq
		\left\lbrace \frac{p}{q}: p \in \set P_{r-1}(\mathbb{F}), \ q \in \set P_{r}(\mathbb{F})
		\right\rbrace,
\end{equation}
where $\set P_r(\mathbb{F})$ denotes polynomials of degree $r$ with coefficients in the field $\mathbb{F}$.
There is a surjective mapping between state-space transfer functions~\cref{eq:ROM_tf}
and elements of $\set R_r(\mathbb{F})$;
hence for any rational function $H_r \in \set R_r(\mathbb{F})$
we can nonuniquely identify a state-space system with matrices $\ma A \in \mathbb{F}^{r\times r}$,
and $\ve b, \ve c\in \mathbb{F}^r$ whose transfer function is $H_r$.
Not all elements of $\set R_r(\mathbb{F})$ are members of $\Htwo$.
For $H_r\in \set R_r(\mathbb{F})$ to be in $\Htwo$
all its poles must lie in the open left half plane;
we denote this space as
\begin{equation}
	\set R_r^+(\mathbb{F}) \! \coloneqq \!
	\set R_r(\mathbb{F}) \cap \Htwo
	=
	\left\lbrace \frac{p}{q}: p \in \set P_{r-1}(\mathbb{F}), \ q \in \set P_{r}(\mathbb{F}),
		\text{ roots of $q$ in LHP }
	\right\rbrace.
\end{equation}
As many applications of model reduction in science and engineering 
involve real dynamical systems---%
e.g., diffuse optimal tomography \cite{SGKCBO15,OKSG17}, 
structural mechanics \cite{G04,SV16}, thermal dynamics \cite{BWG08}, and optimal control \cite{G04}---%
here we focus our attention to finding real reduced order models $H_r \in \set R_r^+(\R)$.


\subsection{Advantages}
There are two main advantages to our \emph{Projected $\Htwo$} (\PHtwo) approach
compared to existing methods for $\Htwo$-optimal model reduction.

Unlike other $\Htwo$-optimal model reduction techniques,
\PHtwo\ requires only evaluations of the full order model transfer function $H(\mu_j)$---%
this places minimal requirements on the full order model.
In contrast, the Iterative Rational Krylov Algorithm (IRKA)~\cite{GAB08}
requires a state-space representation of the full order model.
Other methods require transfer function derivatives;
e.g., Transfer Function IRKA (TF-IRKA)~\cite{BG12} requires first derivatives
and Newton methods~\cite{BG09a,Mei65}
require both first and second derivatives.
Quadrature-based Vector Fitting (QuadVF)~\cite{DGB15}
does not require derivatives,
but does need a particular limit of $H$.


Projected $\Htwo$ typically yields similar or better reduced order models
compared to other techniques while simultaneously user fewer evaluations of the full order model.
This is significant because for large scale models,
evaluating $H(z)$ dominates the cost;
e.g., if $H(z)$ is the transfer function of a large state-space system
\begin{gather}\label{eq:statespace}
\left\lbrace
\begin{aligned}
	\ve x'(t) &= \ma A \ve x(t) + \ve b u(t), \quad \ve x(0) = \ve 0 \\
	y(t) &= \ve c^* \ve x(t)
\end{aligned}
\right\rbrace, 
\quad
\text{where}
\quad 
\begin{aligned}
\ma A \in \C^{n \times n}, \  \ve b, \ve c \in \C^n, \\
H(z) = \ve c^*[z\ma I - \ma A]^{-1} \ve b,
\end{aligned}
\end{gather}
then the cost of evaluating $H(z)$ is dominated by 
the linear solve\linelabel{line:linsolve}.
As the numerical experiments in \cref{sec:results} demonstrate, 
our \PHtwo\ approach uses fewer transfer function evaluations
than both IRKA and TF-IRKA---often by an order of magnitude.
This is because each iteration of \PHtwo\ recycles previous evaluations of $H(\mu_j)$,
whereas each iteration of standard IRKA and TF-IRKA discards
these evaluations and requires $2r$ new evaluations.
Similarly, our algorithm tends to find better local minimizers
because each inner loop of \PHtwo\ solves an overdetermined rational approximation problem
where we are able to try multiple initializations.
In contrast, the rational interpolant generated by IRKA and TF-IRKA at each step is unique.
Compared to QuadVF,
both \PHtwo\ and QuadVF
solve an overdetermined rational approximation but differ
in how they approximate the $\Htwo$-norm.
QuadVF uses a fixed quadrature rule evaluating $H$ on the imaginary axis
whereas \PHtwo\ projects the $\Htwo$-norm as in~\cref{eq:H2MOR_gen_proj}
using evaluations of $H$ in the right half plane.
As our approximation of the $\Htwo$-norm is dynamically chosen,
it often yields better reduced order models.

%% file: properties.tex
\section{Properties of the \texorpdfstring{$\Htwo$}{H2} Hilbert Space\label{sec:properities}}
To begin, we briefly summarize the properties of the $\Htwo$ Hilbert space essential 
for the development of our algorithm.

\subsection{Evaluation of \texorpdfstring{$\Htwo$}{H2} Norm\label{sec:properties:norm}}
Although the $\Htwo$ norm is defined through an integral in~\cref{eq:H2ip},
when $H$ has a state-space representation~\cref{eq:statespace} we can compute this norm exactly.
In this case the impulse response $h(t)$ and transfer function $H(z)$ are
\begin{equation}\label{eq:FOMss}
	h(t) = \ve c^*\exp[\ma A t]\ve b
	\qquad \text{and} \qquad
	H(z) = \ve c^*[z\ma I - \ma A]^{-1} \ve b.
\end{equation}
Then the $\Htwo$ norm can be computed from either the controllability $\ma W_c$ or observability $\ma W_o$ Gramians.
Both Gramians are the solution to a Lyapunov equation~\cite[Sec.~4.3]{Ant05}:
\begin{align}
	\ma W_c &\coloneqq \int_0^\infty e^{\ma A t} \ve b \ve b^* e^{\ma A^*t} \D t,  & 
		\ma A \ma W_c + \ma W_c \ma A^* &= -\ve b\ve b^*; \\
	\ma W_o &\coloneqq \int_0^\infty e^{\ma A^* t} \ve c \ve c^* e^{\ma At} \D t,  & 
		\ma A^* \ma W_o + \ma W_o \ma A &= -\ve c\ve c^* .
\end{align}
Using either Gramian, the $\Htwo$-norm can be evaluated as~\cite[eq.~(5.28)]{Ant05}
\begin{equation}\label{eq:H2_gramian}
	\|H\|_\Htwo^2 \coloneqq \frac{1}{2\pi} \int_{-\infty}^\infty |H(\iu \omega)|^2 \D \omega
		=\int_0^\infty |h(t)|^2 \D t = \ve c^*\ma W_c \ve c = \ve b^*\ma W_o\ve b.
\end{equation}


\subsection{\!\!Reproducing Kernel Hilbert Space\label{sec:properties:RKHS}}
\emph{\!\!Reproducing kernel Hilbert spaces}~\cite{Aro50}
are Hilbert spaces which contain a kernel $K[\mu]$
that is also the sampling operator, 
i.e., $\langle K[\mu], F\rangle = F(\mu)$.
For the $\Htwo$ Hilbert space this kernel is $v[\mu]$
\begin{equation}\label{eq:rkhs}
	v[\mu](z) \coloneqq (z+\conj{\mu})^{-1},
	\quad \text{where} \quad \mu \in \C_+ \coloneqq \lbrace \mu \in \C: \Re \mu > 0\rbrace.
\end{equation}
The reproducing property follows from the Cauchy integral formula~\cite[Lem.~2.4]{GAB08}
\begin{equation}\label{eq:kern_prop}
	\begin{split}
	\langle v[\mu], H \rangle_\Htwo
	&\!=\! \frac{1}{2\pi} \int_{-\infty}^\infty \conj{v[\mu](\iu \omega)} H(\iu \omega) d\omega
	= \frac{1}{2\pi} \int_{-\infty}^\infty \frac{1}{-\iu \omega+\mu} H(\iu \omega) d\omega\\
	&\!=\! \frac{-1}{2\pi} \int_{-\infty}^\infty \frac{1}{\iu \omega - \mu} H(\iu \omega) d\omega
	= \frac{-\iu}{2\pi} \lim_{R\to\infty} \int_{D_R} \frac{1}{z - \mu} H(z) dz
	= H(\mu),
\end{split}
\end{equation}
where $D_R$ denotes the counterclockwise path along the boundary of the right half disk of radius $R$ 
split along the imaginary axis.

\subsection{Meier-Luenberger Optimality Conditions\label{sec:properties:ML}}
The first order optimality conditions for $\Htwo$ model reduction
by state-space systems are called the \emph{Meier-Luenberger} conditions~\cite{ML67}.
These follow from the optimality conditions~\cref{eq:opt_orig},
but are frequently presented as Hermite interpolation conditions
through use of the reproducing kernel $v[\mu]$.
Here we provide a brief derivation
under the assumption $H_r \in \set R_r^+(\C)$ has only simple poles;
for a derivation with out this constraint, see~\cite{DGA10}.
This assumption allows us to parameterize $H_r$ 
in terms of poles $\ve\lambda\in \C^r$ and residues $\ve\rho \in \C^r$
\begin{equation}\label{eq:pole_residue}
	H_r(z; \ve \lambda, \ve \rho) 
		\coloneqq \sum_{k=1}^r \frac{\rho_k}{z - \lambda_k}
		 = \sum_{k=1}^r \rho_k v[-\conj{\lambda_k}](z),
	\quad \ve\lambda, \ve \rho \in \C^r, \ \lambda_k < 0.
\end{equation}
Taking derivatives with respect to the real and imaginary parts yields:
\begin{align}
	\label{eq:ML1}
	\frac{\partial }{\partial \Re \lambda_k} \| H - H_r(\cdot; \ve\lambda, \ve \rho)\|_\Htwo^2
		&= \phantom{-}2 \Re \left\langle 
			\phantom{\iu}\rho_k v[-\conj{\lambda_k}]', 
			H - H_r(\cdot; \ve\lambda, \ve \rho)
			\right\rangle_\Htwo; \\
	\frac{\partial }{\partial \Im \lambda_k} \| H - H_r(\cdot; \ve\lambda, \ve \rho)\|_\Htwo^2
		&= -2 \Re \left\langle 
			\iu \rho_k v[-\conj{\lambda_k}]', 
			H - H_r(\cdot; \ve\lambda, \ve \rho)
		\right\rangle_\Htwo; \\
	\frac{\partial }{\partial \Re \rho_k} \| H - H_r(\cdot; \ve\lambda, \ve \rho)\|_\Htwo^2
		&= -2 \Re \left\langle 
			\phantom{i} v[-\conj{\lambda_k}],
			H - H_r(\cdot; \ve\lambda, \ve \rho) 
		 \right\rangle_\Htwo; \\
	\label{eq:ML4}
	\frac{\partial }{\partial \Im \rho_k} \| H - H_r(\cdot; \ve\lambda, \ve \rho)\|_\Htwo^2
		&= -2 \Re \left\langle 
			 	\iu v[-\conj{\lambda_k}],
				H - H_r(\cdot; \ve\lambda, \ve \rho)
			 \right\rangle_\Htwo;
\end{align}
where $v[\mu]'$ denotes the derivative of $v[\mu]$, 
$v[\mu]'(z) \coloneqq -(z+\conj{\mu})^{-2}$.
When these derivatives above are all zero, 
$H_r$ satisfies the first order necessary conditions.
Denoting these parameter values as $\hve\lambda$ and $\hve \rho$,
combining terms pairwise and invoking~\cref{eq:rkhs} yields
\begin{align}
	\label{eq:ML1eq}	
	0&= \left\langle 
			v[-\conj{\widehat{\lambda}_k} ]', 
			H - H_r(\cdot; \hve\lambda, \hve \rho)
			\right\rangle_\Htwo
	& \Leftrightarrow && 
	 H'(-\conj{\widehat{\lambda}_k}) &= H_r'(-\conj{\widehat{\lambda}_k}; \hve \lambda, \hve \rho ), \\
	\label{eq:ML2eq} 
	0&= \left\langle 
			v[-\conj{\widehat{\lambda}_k} ], 
			H - H_r(\cdot; \hve\lambda, \hve \rho)
			\right\rangle_\Htwo
	& \Leftrightarrow && 
	H(-\conj{\widehat{\lambda}_k}) &= H_r(-\conj{\widehat{\lambda}_k}; \hve \lambda, \hve \rho ),	
\end{align}
provided $|\widehat{\rho}_k| > 0$.
These conditions require a locally optimal $H_r$
to be a Hermite interpolant of $H$ 
at the reflection of the poles $\hve\lambda$ across the imaginary axis.
This result applies to real reduced order models as well, 
since $\set R_r^+(\R)$ is a subspace of $\set R_r^+(\C)$
\begin{equation}
	\set R_r^+(\R) = \lbrace F \in \set R_r^+(\C): \conj{F(z)} = F(\conj{z}) \  \forall z \in \C \rbrace.
\end{equation}
Further, although this parameterization explicitly excludes higher order poles---%
i.e., $H_r(z) = (z-\lambda)^{-2}$ cannot be expressed in this parameterization---%
in practice this is not a concern.
Rational functions with higher order poles are nowhere dense in $\set R_r^+(\C)$~\cite[p.~2739]{DGA10}
and hence cannot be resolved in finite precision arithmetic.

%% file: existing.tex
\section{Existing Algorithms\label{sec:existing}}
There are a variety of techniques for $\Htwo$-optimal model reduction.
Each of these requires access to different information about the full order model
described by $H$.
For example, both IRKA~\cite{GAB08} and TF-IRKA~\cite{BG12} are fixed point iterations
based on rational interpolants,
but IRKA requires access to a state-space representation of $H$
whereas TF-IRKA only requires access to $H(z)$ and $H'(z)$.
Van Dooren, Gallivan, and Absil propose a similar fixed point iteration
that allows for higher order poles in the reduced order model~\cite{DGA08,DGA10}.
There are also Newton methods that require access to $H(z)$, $H'(z)$, and $H''(z)$,
such as the approach developed by Meier~\cite{Mei65}
and a trust-region approach due to Beattie and Gugercin~\cite{BG09a}.
In addition to these optimal methods, there are also several suboptimal methods
that require only access to $H(z)$.
For example, QuadVF uses a quadrature rule to approximate the $\Htwo$-norm~\cite{DGB15}
and $\Htwo$ pseudo-optimality removes the derivative condition from the Meier-Luenberger conditions
and finds a reduced order model with fixed poles that minimizes the $\Htwo$ norm~\cite{WPL13}.
In this section we briefly summarize three of these algorithms
used as comparisons in \cref{sec:results}: IRKA, TF-IRKA, and QuadVF.

\subsection{IRKA}
The Iterative Rational Krylov Algorithm (IRKA)~\cite{GAB08} builds on earlier work constructing
rational interpolants for full order models given in state-space form~\cite{Gri97,YWS85}.
Given a state-space system~\cref{eq:statespace} with matrices $\ma A$, $\ve b$, and $\ve c$,
a Hermite rational interpolant at points $\lbrace \mu_j\rbrace_{j=1}^r$
can be constructed using rational Krylov spaces $\set W$ and $\set V$ 
\begin{align}
	\label{eq:RKV}
	\set V &= \Range(\ma V) \, =
		\Span\{[\mu_1 \ma I - \ma A]^{-1} \ve b, \ldots,  [\mu_r \ma I - \ma A]^{-1} \ve b \}, & 
		 \ma V^*\ma V &= \ma I; \\
	\label{eq:RKW}
	\set W &= \Range(\ma W) = 
		\Span\{[\mu_1 \ma I - \ma A]^{-*} \ve c, \ldots,  [\mu_r \ma I - \ma A]^{-*} \ve c \}, &
		 \ma W^*\ma W &= \ma I.
\end{align}
Then the reduced order model $H_r(z) = \ve c_r^*[z\ma I - \ma A_r]^{-1} \ve b_r$ with matrices
\begin{equation}
	\label{eq:irkaROM}
		\ma A_r = (\ma W^*\ma V)^{-1}\ma W^*\ma A \ma V, \quad
		\ve b_r = (\ma W^*\ma V)^{-1}\ma W^*\ve b, \quad
		\ve c_r = \ma V^*\ve c
\end{equation}
satisfies the Hermite interpolation conditions at each $\mu_j$~\cite[Cor.~2.2]{GAB08}
\begin{equation}
	H(\mu_j) = H_r(\mu_j), \quad H'(\mu_j) = H_r'(\mu_j), \quad j=1, \ldots, r
\end{equation}
provided $\ma W^*\ma V$ is invertible.
This interpolant satisfies the Meier-Luenberger conditions
when the poles of $H_r$---%
the eigenvalues $\lbrace \lambda_j \rbrace_{j=1}^r$ of $\ma A_r$---%
are the interpolation points flipped across the imaginary axis;
i.e., $\lbrace -\conj{\lambda_j} \rbrace_{j=1}^r= \lbrace \mu_j \rbrace_{j=1}^r$.


\input{alg_irka}

IRKA is a fixed point iteration that uses this Hermite interpolant
to find a state-space reduced order model satisfying the Meier-Luenberger conditions.
As summarized in \cref{alg:irka},
given a set of interpolation points $\lbrace \mu_j \rbrace_{j=1}^r$
IRKA constructs a Hermite rational interpolant using~\cref{eq:irkaROM}
and then the poles of this rational interpolant $\lbrace \lambda_j \rbrace_{j=1}^r$
provide the new interpolation points for the next step, $\mu_j = -\conj{\lambda}_j$.
When this iteration converges, 
the reduced order model $\widehat H_r$ satisfies the Meier-Luenberger conditions;
moreover $\widehat H_r$ is a local minimizer as local maximizers are repellent~\cite[subsec.~7.4.2]{BG17}.
Although IRKA often-times converges in practice,
there are only limited cases where convergence is guaranteed~\cite{FBG12} 
and examples exist where this fixed point iteration does not converge~\cite[subsec.~7.4.2]{BG17}.
For large scale model reduction,
the cost of IRKA is dominated by the linear solves in \cref{eq:RKV} and \cref{eq:RKW}.
There are several modifications to IRKA
that mitigate the cost of these linear solves by, 
for example, 
using inexact linear solves~\cite{BGW12,BG06}, recycling information between iterations~\cite{ASGC12},
and constructing local approximations of the full order model~\cite{CPL16}.

\subsection{TF-IRKA}
A critical limitation of IRKA is the need for a state-space representation of $H$
for constructing the rational Krylov subspaces $\set V$ and $\set W$
used to build the rational interpolant reduced order model.
Transfer function IRKA (TF-IRKA)~\cite{BG12}
removes this constraint by constructing the reduced order model
using a Loewner based approach following the work of Anderson and Antoulas~\cite{AA90}. 
Given a set of interpolation points $\lbrace \mu_j\rbrace_{j=1}^r$,
TF-IRKA builds a Hermite interpolant $H_r(z) = \ve c_r^*[z\ma E_r - \ma A_r]^{-1} \ve b_r$ 
using evaluations of $H(\mu_j)$ and its derivative $H'(\mu_j)$ where
\begin{equation}
	[\ma A_r]_{j,k} = \begin{cases}
			-\frac{\mu_j H(\mu_j) - \mu_k H(\mu_k)}{\mu_j - \mu_k}, & j\ne k; \\
			-[zH(z)]'|_{z=\mu_j}, & j=k;
		\end{cases}
	\quad
	[\ma E_r]_{j,k} = \begin{cases}
			-\frac{H(\mu_j) - H(\mu_k)}{\mu_j - \mu_k}, & j\ne k; \\
			-H'(\mu_j), & j=k;
		\end{cases}
\end{equation}
and $[\ve b_r]_j = [\ve c_r]_j = H(\mu_j)$;
here $[\ma A]_{j,k}$ denotes the $j,k$th entry of $\ma A$.
Then, as in IRKA,  
the interpolation points are updated in a fixed point iteration based on the poles of $H_r$,
here the generalized eigenvalues $\lambda$ of $(\ma A_r$, $\ma E_r)$;
e.g., $\ma A_r \ve x = \lambda \ma E_r \ve x$.
When $H$ has a state-space representation,
the iterations of IRKA and TF-IRKA are identical in exact arithmetic.

\subsection{QuadVF}
An alternative approach taken by Quadrature-based Vector Fitting (QuadVF)~\cite{DGB15}
is to approximate the $\Htwo$ norm using a quadrature rule
and then solve the resulting weighted least squares rational approximation problem.
QuadVF uses Boyd/Clenshaw-Curtis quadrature rule~\cite{B87}
using $n$ evaluations of $H$ plus its limit at $\pm \infty$
and includes a scaling parameter $L>0$
\begin{align}\label{eq:H2_bcc}
	\| H \|_\Htwo^2 = \frac{1}{2\pi} \int_{-\infty}^\infty |H( \iu \omega)|^2 \D \omega
	\approx 
		\frac{ | M_+[H] |^2 + | M_-[H] |^2 }{4L(n + 1)} +
		\sum_{j=1}^n w_j |H(z_j)|^2 \\
	 w_j \!=\! \frac{L}{2(n\!+\!1) \sin^2 (j\pi/(n\!+\!1))},\
	 z_j \!=\! \iu L\cot\left(\frac{j\pi}{n\!+\!1}\right), \
 	M_\pm[H] \!=\!\! \lim_{\omega\to\pm \infty}\!  \iu \omega H(i\omega). 
\end{align}
This yields a diagonally weighted least squares rational approximation problem 
\begin{equation}\label{eq:quadVF}
	\setlength\arraycolsep{-2pt}
	\min_{H_r \in \set R_r^+(\R)} \left\|
		\begin{bmatrix}
		\sqrt{w_1} & & & & \\
		& \ddots & & \\
		& & \sqrt{w_n} & & \\
		& & & \!\!\sqrt{w_+} & \\
		& & & & \!\! \sqrt{w_-} 
		\end{bmatrix}
		\! 
		\left(
		\begin{bmatrix}
			H(z_1) \\ \vdots \\ H(z_n) \\ M_+[H] \\ M_-[H]
		\end{bmatrix}
		\!-\!
		\begin{bmatrix}
			H_r(z_1) \\ \vdots \\ H_r(z_n) \\ M_+[H_r] \\ M_-[H_r]
		\end{bmatrix}
		\right)
		\right\|_2 \!; \ 
		 w_{\pm} \!=\! \frac{1}{4L(n\!+\!1)}. \!\!
\end{equation}
A modified Vector Fitting~\cite{GS99} 
constructs a rational approximation in barycentric form
minimizing~\cref{eq:quadVF}
by iteratively updating the nodes of the barycentric representation.
Although QuadVF will necessarily yield 
reduced order models that do not satisfy the Meier-Luenberger conditions
due to the discretization of the $\Htwo$-norm,
this technique frequently yields reduced order models
with small residual norm $\| H - H_r\|_\Htwo$.

%% file: alg_irka.tex
\begin{algorithm}[t]
\begin{minipage}{\linewidth}
\begin{algorithm2e}[H]
	\Input{FOM system $\lbrace \ma A, \ve b, \ve c\rbrace$, 
		initial interpolation points $\ve \mu^{(0)} = [\mu_1^{(0)}, \ldots, \mu_r^{(0)}] $}
	\Output{ROM system $\lbrace \ma A_r, \ve b_r, \ve c_r\rbrace$}
	\For{$\ell=0,1,2,\ldots$}{
		Construct rational Krylov spaces $\ma V$ and $\ma W$ given $\ve\mu^{(\ell)}$ using 
			\cref{eq:RKV} and \cref{eq:RKW}\;
		Build state-space reduced order model $\lbrace \ma A_r, \ve b_r, \ve c_r\rbrace$ 
			with $\ma V$ and $\ma W$ via~\cref{eq:irkaROM}\;
		Choose $\mu_j^{(\ell+1)} = -|\Re\lambda_j| - \iu \Im \lambda_j$, 
			where $\lbrace \lambda_j \rbrace_{j=1}^r$ are eigenvalues of $\ma A_r$\;
	}
\end{algorithm2e}
\vspace{-1.5em}
\end{minipage}
\caption{Iterative Rational Krylov Algorithm (IRKA) (simplified)}
\label{alg:irka}
\end{algorithm}

%% file: projection.tex
\section{Projected Nonlinear Least Squares\label{sec:projection}}
Our approach for $\Htwo$-optimal model reduction 
extends the \emph{projected nonlinear least squares} framework
introduced by Hokanson for finite-dimensional problems~\cite{Hok17}.
In its original presentation, given a nonlinear least squares problem
\begin{equation}\label{eq:pnls_finite}
	\min_{\ve\theta \in \R^q} \| \ve f(\ve\theta) - \tve y\|_2^2 \qquad \ve f: \R^q\to \C^n
\end{equation}
this framework solves a series of projected problems, cf.~\cite[eq.~(2)]{Hok17}:
\begin{equation}
	\ve \theta^{\ell} \coloneqq 
		\argmin_{\ve\theta \in \R^q} \|\ma P_{\set W_\ell} [ \ve f(\ve\theta) - \ve y]\|_2^2
		= \argmin_{\ve\theta \in \R^q} \|\ma W_{\ell}^* [ \ve f(\ve\theta) - \ve y]\|_2^2,
\end{equation} 
where $\ma P_{\set W_\ell}$ is an orthogonal projector onto $\set W_\ell \subset \C^n$
with orthonormal basis $\ma W_\ell\in \C^{n\times m_\ell}$
so that $\ma P_{\set W_\ell} = \ma W_\ell \ma W_\ell^*$.
Then by choosing subspaces $\set W_\ell$ such that the largest subspace angle between
$\set W_\ell$ and the range of the Jacobian of $\ve f(\ve\theta^{\ell})$ is small,
we obtain an accurate solution the original problem~\cref{eq:pnls_finite}.
The advantage of this approach is that at each step
we solve a nonlinear least squares problem of small dimension $m_\ell \ll n$.
A caveat though is we must be able to evaluate the 
inner product $\ma W_\ell^*\ve f(\ve\theta)$ inexpensively.

Applying this framework to $\Htwo$-optimal model reduction
converts an infinite-dimensional optimization problem into a sequence of
finite dimensional nonlinear least squares problems.
Here the reproducing kernel structure of $\Htwo$ 
provides the requisite inexpensive inner product.
In the following subsections
we build the projector for the $\Htwo$ problem (\cref{sec:projection:projector})
and apply \cref{lem:projection} to identify the desired range of the projector (\cref{sec:projection:necessary}).
Since our choice of projector precludes exact containment,
we quantify the error introduced in terms of the subspace angle (\cref{sec:projection:performance})
and show this error can be made arbitrarily small (\cref{sec:projection:fd}).

\subsection{Projection\label{sec:projection:projector}}
Here we construct the projector $P(\ve\mu)$
to be an orthogonal projector onto the subspace $\set V(\ve \mu)$
that spans the kernel vectors $v[\mu_k]$~\cref{eq:rkhs}:
\begin{equation}\label{eq:setV}
	\set V(\ve\mu) \coloneqq \Span \lbrace v[\mu_k] \rbrace_{k=1}^n \subset \Htwo,
	\quad
	\ve \mu \in \C_+^n \coloneqq \lbrace \ve z \in \C^n | \Re z_k > 0\rbrace.
\end{equation}
To build this projector, we first define the linear operator
$V(\ve\mu):\C^n\to \set V(\ve\mu) \subset \Htwo$ and its adjoint $V(\ve\mu)^*:\Htwo \to \C^n$, 
\begin{equation}\label{eq:Vdef}
	V(\ve\mu) \ve c = \sum_{k=1}^n c_k v[\mu_k],
	\quad 
	V(\ve\mu)^*H = 
		\begin{bmatrix}
			\langle v[\mu_1], H\rangle_\Htwo \\
			\vdots \\
			\langle v[\mu_n], H\rangle_\Htwo
		\end{bmatrix}
		= \begin{bmatrix}
			H(\mu_1) \\ \vdots \\ H(\mu_n)
		\end{bmatrix}
		\eqqcolon H(\ve\mu)
\end{equation}
Above we have invoked the kernel identity~\cref{eq:kern_prop} to evaluate the adjoint.
These two operators $V(\ve\mu)$ and $V(\ve\mu)^*$ satisfy the adjoint identity: 
\begin{equation}
	\langle V(\ve\mu)\ve c, H\rangle_\Htwo =
	\sum_{k=1}^n \conj{c_k}\langle v[\mu_k], H\rangle_\Htwo
	=  \ve c^* H(\ve\mu) = 
	\langle \ve c, V(\ve\mu)^*H\rangle_{\C^n}.
\end{equation}
We now construct the orthogonal projector
$P(\ve\mu): \Htwo \to \set V(\ve\mu)$ as
\begin{equation}
	P(\ve\mu) \coloneqq V(\ve\mu) [V(\ve\mu)^* V(\ve \mu)]^{-1} V(\ve \mu)^* =
		V(\ve\mu) \ma M(\ve\mu)^{-1} V(\ve\mu)^*,
\end{equation} 
where $\ma M(\ve\mu)$ is the positive definite Cauchy matrix
\begin{equation}\label{eq:Mdef}	
	[\ma M(\ve\mu)]_{j,k} \coloneqq \langle v[\mu_j], v[\mu_k] \rangle_\Htwo = (\mu_j + \conj{\mu_k})^{-1}.
\end{equation}
With these definitions, the projected norm
can be evaluated as a weighted Euclidean norm 
of the difference between samples of $H$ and $H_r$ evaluated at $\ve\mu$:
\begin{equation}
	\begin{split}
	\| P(\ve\mu)[ H - H_r] \|_\Htwo^2 & = 
		\langle V(\ve\mu) \ma M(\ve\mu)^{-1} V(\ve\mu)^*[ H - H_r], [H - H_r] \rangle_\Htwo \\
		&= \langle \ma M(\ve\mu)^{-\frac12}V(\ve\mu)^*[ H - H_r], 
			\ma M(\ve\mu)^{-\frac12} V(\ve\mu)^*[H  -  H_r] \rangle_\Htwo \\
		\label{eq:PtoNLS}
		&= \langle 
			\ma M(\ve\mu)^{-\frac12} [H(\ve\mu) - H_r(\ve\mu)],
			\ma M(\ve\mu)^{-\frac12} [H(\ve\mu) - H_r(\ve\mu)]
			\rangle_{\C^n} \\
		&= \| \ma M(\ve\mu)^{-\frac12} [H(\ve\mu) -  H_r(\ve\mu)]\|_2^2, 
\end{split}
\end{equation}
where $\ma M(\ve\mu)^{-\frac12}$ is the Hermitian square root of $\ma M(\ve\mu)^{-1}$.
\linelabel{line:square_root}
Although analytically convenient,
computationally we use Cholesky decomposition described in \cref{sec:ratfit:weight}
in place of $\ma M(\ve\mu)^{-\frac12}$.

\subsection{First Order Necessary Conditions\label{sec:projection:necessary}}
We now revisit \cref{lem:projection}
to derive conditions under which a local optimizer of the projected problem
is also a local optimizer of the original problem.
For simplicity we consider a parameterization of $\set R_r^+(\C)$
in terms of poles and residues (cf.~\cref{eq:pole_residue})
\begin{equation}\label{eq:H2MOR_param}
	\hve \lambda, \hve \rho \coloneqq 
		\argmin_{\substack{ \ve \lambda,\ve \rho \in \C^r \\ \Re \lambda_k <0 } } 
		\|H - H_r(\cdot; \ve \lambda, \ve \rho )\|_\Htwo^2, 
		\quad
	H_r(z; \ve\lambda, \ve \rho) \coloneqq \sum_{k=1}^r \rho_k v[-\conj{\lambda_k}](z).
\end{equation}
As $\ve\lambda$ and $\ve \rho$ are contained in an open set,
we can apply \cref{lem:projection}.
Thus if $\hve \lambda, \hve \rho$ are local optimizers of the projected version of~\cref{eq:H2MOR_param},
and $P(\ve\mu) \supseteq \set T(\hve \lambda, \hve\rho)$,
where $\set T(\hve\lambda, \hve \rho)$ contains the derivatives with respect to this parameterization: 
\begin{equation}
	\set T(\ve\lambda, \ve \rho) \coloneqq
		\Span \left\lbrace
			\rho_k v[-\conj{\lambda_k}]', i \rho_k v[-\conj{\lambda_k}]',
			v[-\conj{\lambda_k}], i v[-\conj{\lambda_k}]
		\right\rbrace_{k=1}^r
		\subset \Htwo,
\end{equation}
then $\hve\lambda, \hve \rho$ are also local optimizers of the original problem~\cref{eq:H2MOR_param}.
Note these the vectors in $\set T(\ve\lambda, \ve \rho)$ are precisely
those that appear in the derivation the Meier-Luenberger conditions in \crefrange{eq:ML1}{eq:ML4}.
As the span contains all linear combinations, 
we denote this subspace using $\ve\lambda$ alone
\begin{equation}\label{eq:setMdef}
	\set T(\ve\lambda) \coloneqq \Span \lbrace v[-\conj{\lambda_j}], v[-\conj{\lambda_j}]' \rbrace_{j=1}^{r}
	\supseteq \set T(\ve \lambda, \ve \rho)
\end{equation}
with equality holding when $\rho_k\ne 0$.
Finally, as $\set R_r^+(\R) \subset \set R_r^+(\C)$, 
satisfying $P(\ve\mu) \supseteq \set T(\ve\lambda)$ is sufficient 
to apply \cref{lem:projection} for $\set R_r^+(\R)$.

\subsection{Approximating Necessary Conditions\label{sec:projection:performance}}
Unfortunately the construction of projector $P(\ve\mu)$
precludes exactly satisfying the containment $\Range P(\ve\mu) \supseteq \set T(\ve\lambda)$;
$\set T(\ve \lambda)$ contains not only $v[-\conj{\lambda}_k]$,
but $v[-\conj{\lambda}_k]'$ which is not in the range of $P(\ve\mu)$.
However if the range of $P(\ve\mu)$ approximates $\set T(\ve\lambda)$ 
as measured by the subspace angle
we are able to approximately satisfy the necessary conditions for the original problem.

\subsubsection{Defining Subspace Angles}
Subspace angles on $\Htwo$ are analogously defined to the finite dimensional case.
Given two finite-dimensional subspaces $\set X\subset \Htwo$ and $\set Y\subset \Htwo$ of dimension $m$ and $n$,
the $k$th subspace angle $\phi_k$ between these subspaces $\set X$ and $\set Y$
is (cf.~\cite[eq.~(2)]{BG73}):
\begin{equation}
	\cos \phi_k(\set X, \set Y) \coloneqq 
		\!\! \max_{\substack{X \in \set X \\ \|X \|_\Htwo = 1 \\ \langle X_j, X\rangle_\Htwo = 0 \\ j < k}} \
		\max_{\substack{Y \in \set Y \\ \|Y \|_\Htwo = 1 \\ \langle Y_j, Y\rangle_\Htwo = 0 \\ j < k}}
		\!\! \langle X, Y \rangle_\Htwo,
		 \quad 
	\phi_1 \le \phi_2\le \cdots \le \phi_{\min (m,n)},
\end{equation}
where $X_k\in \set X$ and $Y_k\in \set Y$ are the arguments that yield $\phi_k$.
Given unitary basis operators $B_\set X:\C^m\to \set X$ and $B_\set Y:\C^n\to \set Y$,
we can compute the subspace angles via the singular values of a finite dimensional matrix
\begin{equation}\label{eq:angle:svd}
	\cos \phi_k(\set X, \set Y) = \sigma_k(B_\set X^* B_\set Y);
	\quad B_\set X^* B_\set Y\in \C^{m\times n}
\end{equation}
where $\sigma_k$ denotes the $k$th singular value in descending order; cf.~\cite[Thm.~1]{BG73}.

\subsubsection{Computing Subspace Angles\label{sec:projection:performance:compute}}
To compute the subspace angle between $\set T(\ve\lambda)$ and $\set V(\ve\mu)$
(the range of $P(\ve\mu)$)
we introduce the orthogonal projector $Q(\ve\lambda)$ onto $\set T(\ve\lambda)$.
Defining $V'(\ve\mu):\C^n \to \Htwo$
analogously to $V(\ve\mu)$ with  $v[\mu]'$ in place of $v[\mu]$,
\begin{equation}\label{eq:Vpdef}
	V'(\ve\mu) \ve c \coloneqq \sum_{k=1}^n c_k v[\mu_k]',
	\quad 
	V'(\ve\mu)^*H \coloneqq 
		\begin{bmatrix}
			\langle v[\mu_1]', H\rangle_\Htwo \\
			\vdots \\
			\langle v[\mu_n]', H\rangle_\Htwo
		\end{bmatrix}
		=
		\begin{bmatrix}
			-H'(\mu_1) \\ 
			\vdots \\ 
			-H'(\mu_n)
		\end{bmatrix},
\end{equation}
where the last equality follows from integration by parts.
Then we define the projector $Q(\ve\lambda):\Htwo \to \set T(\ve\lambda)$ as
\begin{equation}
	\setlength{\arraycolsep}{3pt}
	Q(\ve\lambda) \coloneqq
		\begin{bmatrix}
			V(-\conj{\ve \lambda}) & V'(-\conj{\ve\lambda})
		\end{bmatrix}
		\begin{bmatrix}
			V(-\conj{\ve\lambda})^* V(-\conj{\ve\lambda}) & 
			V(-\conj{\ve\lambda})^* V'(-\conj{\ve\lambda}) \\ 
			V'(-\conj{\ve\lambda})^* V(-\conj{\ve\lambda}) &
			V'(-\conj{\ve\lambda})^* V'(-\conj{\ve\lambda})  
		\end{bmatrix}^{\!-1}\!\!
		\begin{bmatrix}
			V(-\conj{\ve\lambda})^* \\ V'(-\conj{\ve\lambda})^*	
		\end{bmatrix}.
\end{equation}
We denote this interior matrix as $\hma M(\ve\lambda)$:
\begin{gather}\label{eq:mhat_def}
	\hma M(\ve\lambda) \coloneqq
		\begin{bmatrix}
			V(-\conj{\ve\lambda})^* V(-\conj{\ve\lambda}) & 
			V(-\conj{\ve\lambda})^* V'(-\conj{\ve\lambda}) \\ 
			V'(-\conj{\ve\lambda})^* V(-\conj{\ve\lambda}) &
			V'(-\conj{\ve\lambda})^* V'(-\conj{\ve\lambda})  
		\end{bmatrix}
		\in \C^{(2r)\times (2r)},
\end{gather}
whose blocks are 
\begin{align}
	\label{eq:mhat1}
	[V\phantom{'}(-\conj{\ve\lambda})^*V\phantom{'}(-\conj{\ve\lambda})]_{j,k} &= 
		\langle v[-\conj{\lambda_j}]\phantom{'},  v[-\conj{\lambda_k}]\phantom{'} \rangle_\Htwo 
		= -(\conj{\lambda_j} + \lambda_k)^{-1}; \\
	\label{eq:mhat2}
	[V\phantom{'}(-\conj{\ve\lambda})^*V'(-\conj{\ve\lambda})]_{j,k} &=
		\langle v[-\conj{\lambda_j}]\phantom{'}, v[-\conj{\lambda_k}]' \rangle_\Htwo 
		= -(\conj{\lambda_j} + \lambda_k)^{-2}; \\
	\label{eq:mhat3}
	[V'(-\conj{\ve\lambda})^*V'(-\conj{\ve\lambda})]_{j,k} &=
		\langle v[-\conj{\lambda_j}]', v[-\conj{\lambda_k}]' \rangle_\Htwo 
		= -2(\conj{\lambda_j} + \lambda_k)^{-3}.
\end{align}
Then to compute the subspace angle, we 
define the unitary basis operators for $\set V(\ve\mu) =\Range P(\ve\mu)$
and $\set T(\ve\lambda) = \Range Q(\ve\lambda)$ 
\begin{align}
	B_{\set V(\ve\mu)}&: \C^n \to \set V(\ve\mu) &
	B_{\set V(\ve\mu)} &\coloneqq V(\ve\mu)\ma M(\ve\mu)^{-\frac12}, \\
	B_{\set T(\ve\lambda)}& : \C^{2r} \to \set M(\ve\lambda) &
	B_{\set T(\ve \lambda)} &\coloneqq 
		\begin{bmatrix}
			V(-\conj{\ve\lambda}) & V'(-\conj{\ve\lambda})
		\end{bmatrix}
		\hma M(\ve\lambda)^{-\frac12}.
\end{align}
Then the largest subspace angle between $\set V(\ve\mu)$ and $\set T(\ve\lambda)$ is 
\begin{equation}\label{eq:cos_max}
	\cos \phi_{\max}(\set T(\ve\lambda), \set V(\ve\mu))
		= \sigma_{\min}(B_{\set V(\ve\mu)}^*B_{\set T(\ve\lambda)}) 
\end{equation}
where $\sigma_{\min}$ denotes the smallest singular value
and 
\begin{equation}\label{eq:compute_angle}
	\setlength{\arraycolsep}{3pt}	
	B_{\set V(\ve\mu)}^* B_{\set T(\ve\lambda)} \!=
		\ma M(\ve\mu)^{-\frac12}  
		\begin{bmatrix}
			V(\ve\mu)^* V(-\conj{\ve \lambda})  &
			V(\ve\mu)^* V'(-\conj{\ve \lambda})  
		\end{bmatrix}
		\hma M(\ve\lambda)^{-\frac12} 
		\! \in \C^{n \times (2r)}.\!\!
\end{equation}

\subsubsection{Necessary Conditions}
We can use the projector $Q(\ve\lambda)$ to provide an alternative 
interpretation of the optimality conditions~\cref{eq:opt_orig} and \cref{eq:opt_proj}.
We first rewrite~\cref{eq:opt_orig}
in terms of the projector $Q(\ve\lambda)$:
\begin{align}
	\langle H - H_r, T\rangle_\Htwo = 0 \  \forall T\in \set T(\ve\lambda) \ 
	&\Leftrightarrow 
	&\langle H - H_r, Q(\ve\lambda) F\rangle_\Htwo &= 0  &\forall F\in \Htwo\phantom{,} \\
	&\Leftrightarrow 
	&\langle Q(\ve\lambda)[H - H_r], F\rangle_\Htwo &= 0  &\forall F \in \Htwo,
\end{align}
where the second equivalence follows as $Q(\ve\lambda)$ is an orthogonal projector.
The choice of $F$ that maximizes this second equivalence is
$F= Q(\ve\lambda)[H - H_r]$ upto a scalar multiple.
Hence, 
\begin{equation}
	\langle H - H_r, T\rangle_\Htwo = 0 \quad \forall T\in \set T(\ve\lambda)
	\quad \Leftrightarrow \quad
	\|Q(\ve\lambda)[ H - H_r]\|_\Htwo = 0.
\end{equation}
Applying the same argument to the projected problem, 
the local optimality conditions in~\cref{eq:opt_orig} and \cref{eq:opt_proj} are equivalently 
\begin{align}
	\label{eq:opt_orig_Q}
	\| Q(\hve\lambda)[H - H_r(\cdot; \hve \lambda, \hve \rho)]\|_\Htwo &= 0
	& \text{(original),} \\ 
	\label{eq:opt_proj_Q}
	\| Q(\hve \lambda) P(\ve\mu)[H - H_r(\cdot; \hve\lambda, \hve \rho)]\|_\Htwo &= 0
	& \text{(projected).} 
\end{align}
The following theorem generalizes \cref{lem:projection}
when $\Range P(\ve\mu) \nsupseteq \set T(\ve\lambda)$.

\begin{theorem}\label{thm:opt}
	Suppose $H_r(\cdot; \hve \lambda, \hve\rho)$ is a local minimizer 
	of the projected problem with $\|Q(\hve \lambda)P(\ve\mu)[H - H_r(\cdot; \hve\lambda, \hve \rho)]\|_\Htwo = 0$.
	Then $H_r(\cdot; \hve \lambda, \hve \rho)$ satisfies the first order optimality conditions 
	for the original problem~\cref{eq:H2MOR_param} 
	with error
	\begin{equation}
		\|Q(\hve\lambda)[H - H_r(\cdot; \hve \lambda, \hve \rho) ]\|_\Htwo
		\le \|H - H_r(\cdot; \hve \lambda, \hve \rho) \|_\Htwo 
		\sin \phi_{\max}(\set V(\ve\mu), \set T(\hve \lambda)).
	\end{equation}
\end{theorem}
\begin{proof}
Inserting the identity $I = I - P(\ve\mu) + P(\ve\mu)$ into the left hand side we have
\begin{multline}
	\| Q(\hve\lambda) [H  - H_r(\cdot; \hve\lambda, \hve \rho)] \|_\Htwo
		\le \\
		\| Q(\hve\lambda) (I - P(\ve\mu)) [  H - H_r(\cdot; \hve\lambda, \hve \rho) ]\|_\Htwo
		+ \| Q(\hve\lambda) P(\ve\mu)[H - H_r(\cdot; \hve \lambda, \hve \rho)]\|_\Htwo.
\end{multline}
The second term is zero since $H_r(\cdot; \hve\lambda, \hve \rho)$ satisfies
the local optimality conditions of the projected problem in \cref{eq:opt_proj_Q}.
Bounding this quantity using the induced $\Htwo$-norm yields
\begin{align}\label{eq:bound_ML_norm}
	\| Q(\hve \lambda) [H  - H_r(\cdot; \hve\lambda, \hve \rho)] \|_\Htwo
	\le 
		\| Q(\hve \lambda) (I - P(\ve\mu))\|_\Htwo 
		\| H - H_r(\cdot;\hve\lambda, \hve \rho) \|_\Htwo.
\end{align}
By definition, the first term on the right is the largest subspace angle between 
$\set V(\ve\mu)$ and $\set T(\hve\lambda)$~\cite[eq.~(13)]{BG73}
\begin{equation}\label{eq:angle_sine}
	\| Q(\hve\lambda)(I-P(\ve\mu))\|_\Htwo
	=
	\sin \phi_{\max}(\set V(\ve\mu), \set T(\hve\lambda)).
\end{equation}
\end{proof}

\subsection{Finite-Difference Subspace\label{sec:projection:fd}}
Although we cannot
have exact containment of $\set T(\ve\lambda)$ inside $\set V(\ve\mu)$,
we can get arbitrarily close.
To provide an intuition of how this can happen, 
note that as a derivative, 
$v[\mu]'$ can be approximated by a finite-difference for some small complex $\delta$:
\begin{equation}
	v[\mu]' \approx \frac{ v[\mu+\delta] - v[\mu-\delta]}{2|\delta|} \in \set V([\mu-\delta, \mu+\delta]).
\end{equation}
In the same spirit,
the following theorem shows that the subspace $\set V(\ve\mu)$ with $\mu_j$ clustering near $-\conj\lambda$
can approximate $\set T(\ve\lambda)$ to arbitrary accuracy.
Combined with \cref{thm:opt},
this theorem shows that local optimizers of the projected $\Htwo$ problem
can satisfy the optimality conditions of the original $\Htwo$ problem
to arbitrary accuracy.

\begin{theorem}\label{thm:angle}
	Let $\set V(\ve\mu)$ be a $n$-dimensional subspace of $\Htwo$ with $\ve\mu \in \C_+^n$
	as defined in~\cref{eq:Vdef} where $\ve\mu$ are distinct
	and let $\set T(\ve\lambda)$ be a $2r$-dimensional subspace of $\Htwo$ with $\ve\lambda \in \C_-^r$
	as defined in~\cref{eq:setMdef} where $\ve\lambda$ are distinct.
	If for each $\lambda_k$, there exist entires of $\ve \mu$
	denoted $\mu_{k,1}$, $\mu_{k,2}$, and $\mu_{k,3}$ where 
	$|\mu_{k,t} + \conj{\lambda_k}| \le \epsilon$ for $t=1,2,3$
	and no entry $\mu_{k,t}$ is repeated,
	then there exists a constant $C$ independent of $\epsilon$ such that
	\begin{equation}
		\sin \phi_{\max}(\set T(\ve\lambda), \set V(\ve\mu)) \le C \epsilon.
	\end{equation} 
\end{theorem}
A proof is provided in the appendix.

%% file: outer.tex
\section{Outer Loop: Sampling the Full Order Model\label{sec:outer}}
Having shown in the previous section that
the local optimality conditions can be satisfied to arbitrary accuracy 
using the projected nonlinear least squares framework,
we now design an efficient algorithm constructing 
a sequence of projectors $\lbrace P(\ve\mu^\ell)\rbrace_{\ell=0}^\infty$
to solve the $\Htwo$-optimal model reduction problem.
As we assume the dominant cost is evaluating the full order model---namely evaluating $H(z)$---%
we choose a nested sequence of projectors
\begin{equation}
	\Range P(\ve\mu^\ell) \subset \Range P(\ve\mu^{\ell+1})
	\qquad
	\ve\mu^{\ell+1} \coloneqq \begin{bmatrix} \ve \mu^{\ell} & \ve \mu_\star^\ell \end{bmatrix}
\end{equation}
to reuse previous evaluations of $H(\mu_j)$.
In this section we provide an effective algorithm 
for choosing these new interpolation points $\ve \mu_\star^\ell$.
However there are many choices for $\ve \mu_\star^\ell$ 
that will yield locally optimal models upon convergence
as illustrated by the following theorem.

\begin{theorem}\label{thm:conv}
Let $H\in \Htwo$ and  $n_0, r$ be a positive integers where $n_0 \ge 2r$.
Given an initial $\ve\mu^{0} \in \C_+^{n_0}$, 
let $\ve\mu^{\ell+1} \coloneqq \begin{bmatrix}\ve \mu^\ell & \ve \mu_\star^\ell\end{bmatrix}$
where $\ve\mu_\star^\ell \in \C_+^{n_\ell}$
and the entries of $\ve\mu^\ell$ are distinct for all $\ell$.
Let $\widehat{H}_r^{\ell}$ be a locally optimal solution to the projected problem
\begin{align}
	\label{eq:conv_Hr}
	\widehat{H}_r^{\ell} &\coloneqq
		\!\! \argmin_{H_r \in \set R_r^+(\C)} \!\!
		\| P(\ve\mu^{\ell}) [H - H_r]\|_\Htwo 
	\text{ with } \|Q(\ve \lambda^\ell)) 
		P(\ve\mu^{\ell})[H-\widehat{H}_r^{\ell}] \|_\Htwo \!= \! 0
\end{align}
and $\ve\lambda^\ell$ are the poles of $\widehat H_r^\ell$.
If $\widehat{H}_r^{\ell}\to \widehat{H}_r\in \Htwo$
where $\widehat H_r$ has $r$ distinct poles $\ve\lambda \in \C_-^r$
and the sequence $\lbrace \ve \mu_\star^{\ell}\rbrace_{\ell=0}^\infty$ 
has limit points $-\ove\lambda$
then $\widehat{H}_r$ satisfies the first order necessary conditions of the $\Htwo$ problem~\cref{eq:opt_orig_Q};
i.e., $\widehat{H}_r$ satisfies the Meier-Luenberger conditions.
\end{theorem}

Inspired by IRKA, 
we could satisfy the conditions of this theorem using the poles
of the current iterate flipped across the imaginary axis.
If $\widehat H_r^\ell$ has poles $\ve\lambda^\ell$
then the choice $\ve\mu_\star^\ell = -\ove\lambda^\ell$
has the desired limit points 
as $\ve\lambda^\ell \to \ve\lambda$ if $\widehat H_r^\ell \to \widehat H_r$ (\cref{lem:poles}).
In practice, 
we desire to use fewer evaluations of the full order model.
Hence at each step we add a single new interpolation point;
e.g., $\mu_\star^\ell \in \C_+$.
Here we describe one effective choice 
and its practical modifications.

\subsection{Selecting Interpolation Points\label{sec:outer:selecting}}
As with the IRKA-inspired update,
we choose new interpolation points $\mu_\star^\ell$ 
from the poles of $\widehat H_r^\ell$ flipped across the imaginary axis.
Specifically we will choose $\mu_\star^\ell$ 
to be the flipped pole of $\widehat H_r^\ell$ 
that is furthest away from $\set V(\ve\mu)$ in terms of the subspace angle
\begin{equation}\label{eq:mu_iter}
	\mu_\star^\ell \coloneqq - \conj{\lambda}_\star^\ell 
	\quad \text{where} \quad
	\lambda_\star^\ell \coloneqq  \argmax_{\lambda \in \ve \lambda^\ell} \  
		\sin \phi_{\max} (\set V(\ve\mu^\ell), \set T(\lambda) ),
\end{equation}
where $\ve\lambda^\ell$ are the poles of $\widehat H_r^\ell$.
Provided these choices of $\lambda_\star^\ell$ do not avoid a particular pole of $\widehat H_r^\ell$ asymptotically,
then $\lbrace \mu_\star^\ell\rbrace_{\ell=0}^\infty$ has limit points $-\ove\lambda$
and this choice satisfies the assumptions of \cref{thm:conv}.
These subspace angles can be computed via~\cref{eq:cos_max}:
\begin{multline}
	\setlength\arraycolsep{1pt}
	\!\!\!
	\cos \phi_{\max}(\set V(\ve\mu), \set T(\lambda)) \! = \!
		\sigma_{\min}  \! \left(
			\!\! \ma M(\ve\mu)^{-\frac12}\!\!
			\begin{bmatrix}
				\frac{1}{\mu_1 - \lambda} & \frac{1}{(\mu_1 - \lambda)^{2}} \\
				\vdots & \vdots \\
				\frac{1}{\mu_n - \lambda} & \frac{1}{(\mu_n - \lambda)^{2}} 
			\end{bmatrix}
			\!\!
			\begin{bmatrix}
				\frac{-1}{\conj{\lambda} + \lambda} & \frac{1}{(\conj{\lambda} + \lambda)^{2}} \\	
				\frac{1}{(\conj{\lambda} + \lambda)^{2}} & \frac{-2}{(\conj{\lambda} + \lambda)^{3}} 	
			\end{bmatrix}^{-\frac12}
		\right)\!\!.\!\!\!\!
\end{multline}
Evaluating this subspace angle is inexpensive, taking only $\order(n^2)$ operations,
and can reuse the factorization of $\ma M(\ve\mu)$ 
computed as part of solving the nonlinear least squares problem
as discussed in \cref{sec:ratfit:weight}.

\subsubsection{When \cref{thm:conv} Does Not Hold}
A notable case that does not satisfy the assumptions of this theorem 
occurs when $H\in \set R_r^+(\R)$.
For this case, given any set of distinct samples $\ve\mu  \in \C_+^n$
where $n\ge 2r$ we recover $H$ exactly, namely $\widehat H_r = H$.
Thus the poles of $\widehat H_r^\ell$ are always the same
and the sequence $\lbrace\mu_\star^\ell\rbrace_{\ell=0}^\infty$
violates the assumption of being distinct.
However in this case, 
we still obtain an optimal reduced order model.

\subsubsection{Improved Updates}
Before continuing,
we note there may be better updates than~\cref{eq:mu_iter}.
For example, we could choose $\mu$ to minimize the subspace angle
between the tangent subspace associated with the current iterate and the projection subspace,
\begin{equation}\label{eq:mu:opt}
	\min_{\mu \in \C_+} \ 
	\sin \phi_{\max}( \set T(\ve\lambda(\widehat{H}_r^\ell)), \set V(\ve\mu^\ell) \cup \set V(\mu)).
\end{equation}
However, unlike the update rule we propose,
this is a nonconvex optimization problem 
and evaluating these subspace angles is both expensive and ill-conditioned.

\subsection{Practical Algorithm\label{sec:outer:algorithm}}
In numerical practice, we modify the iteration given in \cref{eq:mu_iter} 
as described in \cref{alg:ph2}.

\input{alg_ph2}

\subsubsection{Conjugate Samples}
Since we assume that the full order model $H\in \Htwo$ is real,
we can evaluate $H(\mu)$ and $H(\conj{\mu})$
through only one evaluation of $H$ as $\conj{H(\mu)} = H(\conj{\mu})$.
Thus if $\mu_\star^\ell$ is not on the real line, 
we automatically include its conjugate in line~\ref{alg:ph2:pole:conj}
which incurs no additional evaluation of $H$.

\subsubsection{Intermediate Dimension}
If the rational approximation problem is underdetermined,
we choose an intermediate dimension for the reduced order model $\widehat{r}<r$
such that the rational approximation problem on line~\ref{alg:ph2:pnls}
exactly determined or overdetermined.
We always pick an even order (cf.\ line~\ref{alg:ph2:dim}) 
as odd real rational models must have a pole on the real line,
but $\widehat H_r$ may not.

\subsubsection{Spurious Poles\label{sec:outer:algorithm:spurious}}
The least squares rational approximation step on line~\ref{alg:ph2:pnls}
may yield a solution whose poles are spurious for our purposes.
For example, a pole may be far away from existing samples;
e.g., with $\Im \mu_j \in [-1,1]$, there could be a pole $\lambda_k$ with  $\Im \lambda_k \approx 10^7$.
In this case as $(\mu_j -\lambda_k)^{-1}$ is small 
leading to inaccurate estimation of $\lambda_k$.
Similarly, a pole $\lambda_k$ may appear on the imaginary axis with $\Re \lambda_k = 0$,
an artifact of the box constraints introduced to ensure $H_r\in \Htwo$
(see \cref{sec:ratfit:param}). 
Without modification, a spurious pole can trigger a cascade leading the algorithm to fail:
the subspace angle criteria on line~\ref{alg:ph2:update}
ensures these spurious poles are selected,
causing $\ma M(\ve\mu)$ to be ill-conditioned,
and leading subsequent rational approximation steps to fail.
To mitigate this failure mode we introduce a heuristic to 
identify and replace these spurious poles.
We label poles as spurious if they fall outside of
box containing $\ve\mu$ that has been enlarged by a factor of ten 
\begin{equation}
	\set F(\ve\mu) \coloneqq [-10\max_{\mu  \in \ve\mu} \Re\mu, -0.1 \min_{\mu \in \ve\mu} \Re\mu]
			\times \iu [10\min_{\mu\in \ve \mu} \Im \mu, 10\max_{\mu \in \ve\mu} \Im \mu]
		\subset \C_- .
\end{equation}
Scaling by a factor of ten balances the need allow
new interpolation points $\mu$ to extend beyond existing $\ve\mu$
while still identifying likely spurious poles.
When a pole has been labeled spurious,
we replace this pole with the corresponding pole from a AAA rational approximation 
as shown on line~\ref{alg:ph2:replace};
note this AAA rational approximation was already constructed during 
the initialization of the rational approximation step (see \cref{sec:ratfit:init}).
Although the AAA poles are not optimal, 
in practice they are not spurious and
provide relevant information about $H$.
Although this alters the iteration presented in \cref{thm:conv},
this result still applies as if $\widehat H_r^\ell\to \widehat H_r$,
eventually all poles lie in $\set F(\ve\mu)$
and this modification is inactive.

\subsubsection{Termination Criteria}
Here we terminate the algorithm when the difference 
between subsequent iterates measured in the $\Htwo$ norm
is small (line~\ref{alg:ph2:stop}),
choosing this norm to remove the effects of parameterization of $H_r$.
Computing this quantity only requires $\order(r^3)$ operations
using~\cref{eq:H2_gramian},
Ideally based on \cref{thm:opt}, we would prefer to terminate when
the subspace angle between $\set V(\ve\mu^\ell)$ and $\set T(\ve\lambda(\widehat H_r^\ell))$
is sufficiently small;
unfortunately this quantity proves difficult to compute accurately.

%% file: alg_ph2.tex
\begin{algorithm}[t]
\begin{minipage}{\linewidth}
\begin{algorithm2e}[H]
	\setlength\arraycolsep{3pt}
	\Input{Real FOM $H \! \in\! \Htwo$, order $r>0$, initial samples $\ve\mu^0 \! \in \! \C_+^{n_0}$, tolerance~$\tau \!>\! 0$}
	\Output{Real ROM $H_r \! \in\! \Htwo$}
	\For{$\ell=0,1,2,\ldots$}{
		Set $n$ to be the length of $\ve\mu^\ell$\;
		\lIf{$n < 2r$\label{alg:ph2:dim}}{
			Set $r \leftarrow 2\lfloor n/4\rfloor$ for only the next iteration 
		}
		Solve projected problem: 
			$\widehat{H}_r^\ell \leftarrow \argmin_{H_r \in \set R_r^+(\R) } \|P(\ve\mu^\ell) [ H - H_r]\|_\Htwo$\;
			\label{alg:ph2:pnls}
		\lIf{$\| \widehat{H}_r^\ell - \widehat{H}_r^{\ell-1}\|_\Htwo <\tau$ \label{alg:ph2:stop}}{ {\bf break} }
		Compute poles of $\widehat{H}_r^\ell$: $\ve\lambda \leftarrow \lambda(\widehat{H}_r^\ell)\subset \C^r_-$ \;
		Using AAA, construct a degree-$(\widehat r, \widehat r)$ rational approximation $G$ of 
			$\ve\mu, H(\ve\mu)$\;
		Compute poles of $G$: $\ve \nu \leftarrow \lambda(G)$; 
			flip into $\C_-^r$ by $\nu_j \leftarrow -|\Re	\nu_j| + \iu \Im \nu_j$\;
		Order $\ve\nu$ to minimize $|\nu_j -\lambda_j|$ using the Kuhn-Munkres algorithm\;
		\lIf{$\lambda_j \notin \set F(\ve\mu)$}{set $\lambda_j \leftarrow \nu_j$ \label{alg:ph2:replace}} 
		Compute furthest pole: $\lambda_\star \leftarrow \argmax_{\lambda \in \ve\lambda} \phi_{\max}(\set M(\lambda), \set V(\ve\mu))$ \;
			\label{alg:ph2:update}
		\lIf{$\Im \lambda_\star\ne 0$}{
			Update samples: $\ve\mu^{\ell+1} \leftarrow \begin{bmatrix} \ve \mu^\ell & -\lambda_\star& -\conj{\lambda_\star} \;\;   \end{bmatrix}$
			\label{alg:ph2:pole:conj}
		}
		\lElse{
			Update samples: $\ve\mu^{\ell+1} \leftarrow \begin{bmatrix} \ve \mu^\ell & -\lambda_\star \end{bmatrix}$ 
		}
	}
\end{algorithm2e}
\end{minipage}
\vspace{-1.5em}
\caption{Projected Nonlinear Least Squares for Real $\Htwo$ Model Reduction}
\label{alg:ph2}
\end{algorithm}

%% file: ratfit.tex
\section{Inner Loop: Constructing a Reduced Order Model\label{sec:ratfit}}
A critical component of our Projected $\Htwo$ approach 
is the construction of a \emph{weighted least squares rational approximant}
on line~\ref{alg:ph2:pnls} of \cref{alg:ph2}
\begin{equation}\label{eq:ratfit_equiv}
	\widehat H_r^\ell  \coloneqq 
		\argmin_{H_r \in \set R_r^+(\R)}
			\| P(\ve\mu^\ell)[H - H_r]\|_\Htwo
		=
		\argmin_{H_r \in \set R_r^+(\R)}
		\|
			\ma M(\ve\mu^{\ell})^{-\frac12} [
			H(\ve\mu^\ell)
			-
			H_r(\ve\mu^\ell)
			]
		\|_2.
\end{equation}
In \cref{thm:angle},
an important assumption required for the outer loop to yield local optimizers
was that these iterates $\widehat H_r^\ell$ are local optimizers of~\cref{eq:ratfit_equiv}.
Unfortunately, many popular rational fitting algorithms do
not yield local optimizers for this problem.
For example, Adaptive Anderson-Antoulas (AAA)~\cite{NST18} exactly interpolates at some $\mu_j$
and suboptimally approximates on the remainder.
Both the Sanathanan-Koerner iteration~\cite{SK63}
and Vector Fitting~\cite{GS99} have been shown, in general, 
to converge to rational approximants that do not satisfy the necessary conditions~\cite{Shi16}.
This lack of existing results motivates our development of a nonlinear least squares approach to
solve this weighted rational approximation problem.
Here we introduce a two-term partial fraction expansion of $H_r$
to implicitly enforce the constraint that $H_r$ is real
and apply variable projection~\cite{GP73}
to reduce the dimension of the optimization problem.
Additional information about this approach appears
in a companion manuscript~\cite{HM18y}.

\subsection{Parameterization\label{sec:ratfit:param}}
To construct a rational approximation,
we must first choose a parameterization of 
the space of real rational functions of degree $(r-1,r)$: $\set R_r^+(\R)$.
As discussed in~\cite[sec.~3.4]{HM18y}, naive approaches have significant drawbacks:
parametrizing $\set R_r^+(\R)$ as a ratio of monomials rapidly leads to ill-conditioning
and a pole-residue parameterization requires additional nonlinear constraints
to ensure the rational function has a real representation.
Instead, we use a two-term partial fraction expansion following~\cite[subsec.~4.2]{HM18y}:
\begin{equation}\label{eq:param}
	H_r^{\text{PF} }(z; \ve a, \ve b) \coloneqq
	\begin{cases}
		\displaystyle \phantom{\frac{a_r}{z + b_r}}\phantom{+}
		\sum_{k=1}^{\lfloor r/2 \rfloor} \frac{ a_{2k} z + a_{2k-1}}{z^2 + b_{2k} z + b_{2k-1}}, & r \text{ even;}\\
		\displaystyle \frac{a_r}{z + b_r}
		+ \! \sum_{k=1}^{\lfloor r/2 \rfloor} \frac{ a_{2k} z + a_{2k-1}}{z^2 + b_{2k} z + b_{2k-1}}, 
		& r \text{ odd;}
	\end{cases}
	\quad \ve a, \ve b\in \R^r.
\end{equation} 
This parameterization has several advantages:
it is (comparably) numerically stable,
requires only $2r$ \emph{real} parameters, 
and allows us to enforce $H_r^{\text{PF}} \in \set R_r^+(\R)$ 
through a box constraint.
Recall $H_r^{\text{PF}} \in \set R_r^+(\R)\subset \Htwo$ if its poles are in the left half plane.
As the poles of $H_r^{\text{PF}}$ are $-b_{2k}/2 \pm \sqrt{b_{2k}^2/4 - b_{2k+1}}$
(and $-b_r$ if $r$ is odd)
it is both necessary and sufficient to require $b_{k}> 0$
for $H_r^{\text{PF}}$ to be in $\set R_r^+(\R)$.

\subsection{Variable Projection}
Using this parameterization, 
we construct the rational approximant
by solving the rational approximation problem~\cref{eq:ratfit_equiv}
\begin{equation}\label{eq:opt1}
	\min_{\substack{ \ve a, \ve b\in \R^r \\ b_k > 0} } \left\|
		\ma M(\ve\mu)^{-\frac12} [H(\ve\mu) - H_r^{\text{PF}}(\ve\mu; \ve a, \ve b)] 
	\right\|_2.
\end{equation}
We  write $H_r^{\text{PF}}(\ve\mu; \ve a, \ve b)$
as the matrix-vector product
$H_r^{\text{PF}}(\ve\mu; \ve a, \ve b) = \ma \Theta(\ve b) \ve a$
where 
\begin{align} \label{eq:theta}
	\ma \Theta(\ve b) &\coloneqq 
	\begin{cases}
		\begin{bmatrix} \ma \Omega([\ve b]_{1:2}) & \cdots & \ma \Omega([\ve b]_{r-1:r})  \end{bmatrix}, 
		& r \text{ even};\\
		\begin{bmatrix} \ma \Omega([\ve b]_{1:2}) & \cdots & \ma \Omega([\ve b]_{r-2:r-1}) &
			(\ve \mu + b_r)^{-1}  \end{bmatrix}, 
		& r \text{ odd};
	\end{cases} \\
	\ma \Omega\left(\begin{bmatrix} b_1 \\ b_2 \end{bmatrix}\right) & \coloneqq
			\begin{bmatrix}
				\frac{\mu_1}{\mu_1^2 + b_2 \mu_1 + b_1} &
				\frac{1}{\mu_1^2 + b_2 \mu_1 + b_1} \\
				\vdots & \vdots \\
				\frac{\mu_n}{\mu_n^2 + b_2 \mu_n + b_1} 
				& \frac{1}{\mu_n^2 + b_2 \mu_n + b_1} 
			\end{bmatrix} \in \C^{n\times 2}.
\end{align}
This exposes the separable structure of~\cref{eq:opt1},
which for fixed $\ve b$ yields a linear least squares problem in $\ve a$:
\begin{equation}\label{eq:opt2}
	\min_{\substack{\ve a, \ve b\in \R^r \\ b_k > 0}} \| \ma M(\ve\mu)^{-\frac12}[H(\ve\mu) - \ma \Theta(\ve b) \ve a]\|_2.
\end{equation}
As the objective function involves the complex $2$-norm,
we recast this as an optimization problem over the real $2$-norm
by splitting into real and imaginary components 
\begin{align}
	\setlength\arraycolsep{3pt}
	\min_{\substack{\ve a, \ve b\in \R^r\\ b_k > 0}}
		\left\| 
			\begin{bmatrix} 
				\Re \ma M(\ve\mu)^{-\frac12} H(\ve\mu) \\ 
				\Im \ma M(\ve\mu)^{-\frac12} H(\ve\mu)
			\end{bmatrix}
		 - 
			\begin{bmatrix} 
				\Re \ma M(\ve\mu)^{-\frac12}\ma \Theta(\ve b) \\
				 \Im \ma M(\ve\mu)^{-\frac12} \ma \Theta(\ve b) 
			\end{bmatrix}
		\ve a \right\|_2.
\end{align}
Applying variable projection 
yields an equivalent optimization problem over $\ve b$ alone
\begin{align}\label{eq:ratfit_vp}
	\min_{\substack{\ve b \in \R^r \\ b_k > 0}}
		\left\|
			\left[ \ma I - 
			\begin{bmatrix} 
				\Re \ma M(\ve\mu)^{-\frac12}\ma \Theta(\ve b) \\
				 \Im \ma M(\ve\mu)^{-\frac12} \ma \Theta(\ve b) 
			\end{bmatrix}
			\begin{bmatrix} 
				\Re \ma M(\ve\mu)^{-\frac12}\ma \Theta(\ve b) \\
				 \Im \ma M(\ve\mu)^{-\frac12} \ma \Theta(\ve b) 
			\end{bmatrix}^+
			\right]
			\begin{bmatrix} 
				\Re \ma M(\ve\mu)^{-\frac12} H(\ve\mu) \\ 
				\Im \ma M(\ve\mu)^{-\frac12} H(\ve\mu)
			\end{bmatrix}
		\right\|_2,
\end{align}
where $^+$ denotes the pseudoinverse.
Computationally we replace $\ma M(\ve\mu)^{-\frac12}$ with a high accuracy pivoted Cholesky factor
described in \cref{sec:ratfit:weight}.
\Cref{alg:ratfit} shows how to construct the residual and Jacobian of~\cref{eq:ratfit_vp}.
Note on line~\ref{alg:ratfit:qr} we use a column-pivoted and row sorted QR
to improve conditioning; see~\cite[sec.~19.4]{Hig02}.

\input{alg_ratfit}

\subsection{Optimization\label{sec:ratfit:opt}}
There are many algorithms for nonlinear least squares problems.
Here we use Branch, Coleman, and Li's trust region algorithm~\cite{BCY99}
as implemented in SciPy's \verb|least_squares|~\cite{scipy}
due to its ability to enforce box constraints.
This requires the residual and Jacobian of~\cref{eq:opt1},
which we compute using \cref{alg:ratfit}.
When $\uma R$ is not invertible,
we terminate the optimization
as the rational approximation has degree less than $r$.

\subsection{Initialization\label{sec:ratfit:init}}
The rational approximation problem often has many local minimizers with large mismatch.
To ensure that we find a good local minimizer with small mismatch,
at each step we try two initializations of the optimization algorithm to solve projected problem~\cref{eq:ratfit_vp},
keeping the one with smaller mismatch.
One initialization uses the poles of the previous iterate when both are of the same dimension.
The other initialization uses the AAA algorithm to construct a degree $(r,r)$ rational approximant
ignoring the weighting.
As these poles will not in general appear in conjugate pairs---%
a requirement of $\set R_r^+(\R)$---%
we use the Kuhn-Munkres algorithm to pair these poles
with their nearest conjugate-pair, average them,
and flip into the left half plane if necessary
to find poles of an element of $\set R_r^+(\R)$.

\subsection{Evaluating the Gram Matrix\label{sec:ratfit:weight}}
Cauchy matrices---such as $\ma M(\ve\mu)$ appearing in~\cref{eq:ratfit_equiv}---%
have a well-deserved reputation for being ill-conditioned.
Our application is no exception.
As the outer loop converges
the entries of $\ve\mu$ become increasingly close 
and consequently $\ma M(\ve\mu)$ becomes increasingly ill-conditioned.
Fortunately the Cauchy matrix structure
enables factorizations with high relative accuracy~\cite{BKO02,Dem00}
which we can use to accurately compute $\| \ma M(\ve\mu)^{-\frac12} \ve z\|_2$ for any $\ve z \in \C^n$.\linelabel{line:sqrt_chol}
Following Demmel~\cite[Alg.~3]{Dem00},
we compute the Cholesky factorization of $\ma M(\ve\mu)$ 
using Gaussian elimination with complete pivoting 
\begin{equation}\label{eq:ldl}
	\ma M(\ve\mu) = \ma P \ma L \ma D\ma L^*\ma P^*
\end{equation} 
where $\ma D$ is a diagonal matrix, $\ma L$ is lower triangular,
and $\ma P$ is a permutation matrix.
As $\ma M(\ve\mu)$ is Hermitian, 
we can perform the necessary pivoting a priori reducing the computational complexity 
of this decomposition from $\order(n^3)$
to $\order(n^2)$ operations~\cite[Alg.~4]{Dem00}.
Then we evaluate this norm as 
\begin{equation}
	\| P(\ve\mu) F\|_\Htwo = 
	\| \ma M(\ve\mu)^{-\frac12}F(\ve\mu)\|_2 = 
		\| \ma D^{-\frac12}\ma L^{-*} \ma P F(\ve\mu)\|_2
	\quad \forall F \in \Htwo.
\end{equation}
For completeness, \cref{alg:ldl} describes how to compute this high relative accuracy Cholesky factorization
where $\odot$ denotes the Hadamard (entry-wise) product.

\input{alg_ldl}

%% file: alg_ratfit.tex
\begin{algorithm}[t]
\begin{minipage}{\linewidth}
\begin{algorithm2e}[H]
	\Input{Parameters $\ve b \in  \R^r$, factorization $\ma M(\ve\mu) = \ma P\ma L \ma D \ma L^*\ma P^*$~\cref{eq:ldl}}
	\Output{Residual $\uve r\in \R^{2n}$ and Jacobian $\uma J\in \R^{(2n)\times r}$}
	Form $\ma \Theta \leftarrow \ma \Theta(\ve b)$ as given in \cref{eq:theta}\;
	Compute the short form QR decomposition
	 $\uma Q \uma R \leftarrow \begin{bmatrix} \Re \ma D^{-1/2}\ma L^{-*}\ma P \ma \Theta \\ 
			\Im \ma D^{-1/2}\ma L^{-*} \ma P \ma \Theta \end{bmatrix}$\; \label{alg:ratfit:qr}
	Define $\uve h \leftarrow \begin{bmatrix} \Re \ma D^{-1/2}\ma L^{-*}\ma P H(\ve\mu) \\ 
		\Im \ma D^{-1/2}\ma L^{-*} \ma P H(\ve\mu) \end{bmatrix}$ \;
	Compute (real) residual $\uve r \leftarrow \uve h - \uma Q \, \uma Q^\trans \uve h$ \; 
	Form complex residual $\ve r \leftarrow [\uve r]_{1:n} + i[\uve r]_{n+1:2n} $\;
	Compute linear coefficients $\ve a 	\leftarrow \uma R^+\uma Q^\trans \uve h$\;
	\For{$k=1,\ldots, \lfloor r/2\rfloor$}{
		$\ve d \leftarrow \ve\mu^2 + b_{2k}\ve\mu + b_{2k-1}$\;
		$[\uma K]_{\cdot,2k-1} \leftarrow 
		 [\ma I - \uma Q\, \uma Q^\trans] 
		\begin{bmatrix}
		 \Re \ma D^{-1/2}\ma L^{-*} \ma P \diag(\ve d)^{-2} (a_{2k} \ve\mu + a_{2k-1}) \\
		 \Im \ma D^{-1/2}\ma L^{-*} \ma P \diag(\ve d)^{-2} (a_{2k} \ve\mu + a_{2k-1}) \\
		\end{bmatrix}$\;
		$[\uma K]_{\cdot,2k} \leftarrow 
		 [\ma I - \uma Q\, \uma Q^\trans] 
		\begin{bmatrix}
		 \Re \ma D^{-1/2}\ma L^{-*} \ma P \diag(\ve\mu)\diag(\ve d)^{-2} (a_{2k} \ve\mu + a_{2k-1}) \\
		 \Im \ma D^{-1/2}\ma L^{-*} \ma P \diag(\ve\mu)\diag(\ve d)^{-2} (a_{2k} \ve\mu + a_{2k-1}) \\
		\end{bmatrix}$\;
		$[\uma L]_{\cdot,2k-1} 
			\leftarrow \uma Q\uma R^{+\trans} 
			\begin{bmatrix}
				 \Re \diag(\ve d)^{-2*}\ma P^* \ma L^{-1}\ma D^{-1/2}\ve r \\
				 \Im \diag(\ve d)^{-2*}\ma P^* \ma L^{-1}\ma D^{-1/2}\ve r
			\end{bmatrix}$\;
		$[\uma L]_{\cdot,2k} 
			\leftarrow \uma Q\uma R^{+\trans} 
			\begin{bmatrix}
				 \Re \diag(\ve \mu)\diag(\ve d)^{-2*}\ma P^*\ma L^{-1}\ma D^{-1/2}\ve r \\
				 \Im \diag(\ve \mu)\diag(\ve d)^{-2*}\ma P^*\ma L^{-1}\ma D^{-1/2}\ve r
			\end{bmatrix}$\;
	}
	\If{$r$ is odd}{
		$[\uma K]_{\cdot,r} \leftarrow
			[\ma I - \uma Q\, \uma Q^\trans] 
			\begin{bmatrix}
			 \Re \ma D^{-1/2}\ma L^{-*} \ma P \diag(\ve \mu+ b_r)^{-2} a_r \\
		 	\Im \ma D^{-1/2}\ma L^{-*} \ma P \diag(\ve \mu + b_r)^{-2} a_r 
		\end{bmatrix}$\;
		$[\uma L]_{\cdot,r} 
			\leftarrow \uma Q\uma R^{+\trans} 
			\begin{bmatrix}
				 \Re \diag(\ve \mu + b_r)^{-2*}\ma P^*\ma L^{-1} \ma D^{-1/2}\ve r \\
				 \Im \diag(\ve \mu + b_r)^{-2*}\ma P^*\ma L^{-1} \ma D^{-1/2}\ve r
			\end{bmatrix}$\;
	}
	$\uma J \leftarrow \uma K + \uma L$\;  
\end{algorithm2e}
\end{minipage}
\vspace{-1.5em}
\caption{Residual and Jacobian for Real Partial Fraction Parameterization}
\label{alg:ratfit}
\end{algorithm}

%% file: alg_ldl.tex
\begin{algorithm}[t]
\begin{minipage}{\linewidth}
\begin{algorithm2e}[H]
\Input{$\ve\mu \in \C^n$ determining $[\ma M(\ve\mu)]_{j,k} = (\mu_j + \conj{\mu_k})^{-1}$}
\Output{Permutation matrix $\ma P\in \R^{n\times n}$, 
		lower triangular matrix $\ma L \in \C^{n\times n}$, and
		diagonal matrix $\ma D\in \R^{n\times n}$
		such that $\ma M(\ve\mu) = \ma P\ma L\ma D\ma L^*\ma P^*$
	}
	$\ma P = \ma I\in \R^{n\times n}$\;
	$\ve s \leftarrow (\ve \mu + \conj{\ve\mu} )^{-1}$\;
	\For{$k=1,\ldots, n$}{
		$j \leftarrow \argmax_{j=k,\ldots n} s_j$\;
		Permute $\ma P_{\cdot,k},\ma P_{\cdot, j}\leftarrow \ma P_{\cdot,j}, \ma P_{\cdot,k}$\;
		Permute $\mu_j,\mu_k \leftarrow \mu_k,\mu_j$\;
		Permute $s_j, s_k \leftarrow s_k, s_j$\;
		$[\ve s]_{k+1:n}\leftarrow
				|[\ve\mu]_{k+1:n} - \mu_k|^2\odot|[\ve\mu]_{k+1:n} +\conj{\mu_k}|^{-2}
		$
	}
	$\ve g \leftarrow \ve 1 \in \C^n$\;
	\For{$k=1,\ldots,n-1$}{
		$[\ma D]_{k,k} \leftarrow \mu_k + \conj{\mu_k}$\;
		$[\ma L]_{k:n,k} \leftarrow  [\ve g]_{k:n} \odot ([\ve \mu]_{k:n} + \conj{\mu_k})^{-1}$\;
		$[\ve g]_{k+1:n} \leftarrow 
			[\ve g]_{k+1:n} \odot ( [\ve\mu]_{k+1:n} - \mu_k) \odot ([\ve\mu]_{k+1:n} + \conj{\mu_k})^{-1}$\;
	}
	$[\ma D]_{n,n} \leftarrow (\mu_n + \conj{\mu_n})^{-1}$\;
	$[\ma L]_{n,n} \leftarrow g_n$\;
\end{algorithm2e}
\end{minipage}
\vspace{-1.5em}
\caption{Pivoted $\ma L \ma D\ma L^*$ factorization of $\ma M(\ve\mu)$}
\label{alg:ldl}
\end{algorithm}

%% file: results.tex
\section{Numerical Experiments\label{sec:results}}
Here we present numerical experiments comparing our
Projected $\Htwo$ approach to IRKA, TF-IRKA, and QuadVF.
Following the principles of reproducible science,
our code implementing these algorithms and generating the figures is available at 
{\tt \url{https://github.com/jeffrey-hokanson/SYSMOR}}.

\subsection{Test Problems}
Systems with oscillatory behavior prove challenging for model reduction.
Our two test problems have significant oscillatory behavior
as illustrated by the large number of peaks in their Bode plots 
in Figures~\ref{fig:iss} and \ref{fig:delay}.

\subsubsection{ISS 1R Component}
The 1R component of the International Space Station 
is a standard test problem for model reduction~\cite[subsec.~2.11]{CD02}
with a sparse state-space representation of dimension $270$ with $405$ nonzeros.
Here we consider the $(1,1)$ block of this transfer function.

\subsubsection{Delay System}
Here we slightly modify the delay system from~\cite[sec.~5]{BG09}:
\begin{equation}
	\label{eq:delay}
	\begin{split}
	&H(z) = \ve c^\trans ( z \ma E - \ma A_0 - e^{-\tau z} \ma A_1)^{-1}\ve b \qquad \text{where}\\
	& 
	\ma E \coloneqq \frac{2}{\sqrt{\epsilon}}\ma I + \ma T, \quad
	\ma A_0 \coloneqq 
		\frac{2 + 2\rho}{\tau \rho} \left(
			\ma T - 	
			\frac{2}{\sqrt{\epsilon}} \ma I
		\right), \quad  
	\ma A_1 \coloneqq
		\frac{2- 2\rho}{\tau \rho} \left(
			\ma T - 	
			\frac{2}{\sqrt{\epsilon}} \ma I
		\right),
	\end{split}
\end{equation}
$\ma I\in \R^{n\times n}$ is the identity matrix
and $\ma T\in \R^{n\times n}$ has ones on the first superdiagonal, first subdiagonal,
and the $(1,1)$ and $(n,n)$ entries, and is zero otherwise.
Here we choose $n=1000$, $\tau=1$, $\rho = 0.1$, $\epsilon = 0.01$,
pick $\ve b$ such that its first two entries are one and the remainder zero,
and set $\ve c = \ve b$.

\subsection{Setup\label{sec:results:setup}}
The performance of each algorithm strongly depends
on both the choice of the initialization 
and the termination conditions.
Here we describe how we construct comparable  
conditions for each algorithm.

\subsubsection{Initialization\label{sec:results:setup:init}}
For Projected $\Htwo$, IRKA, and TF-IRKA
each requires an initial set of interpolation points.
Following standard practice for IRKA,
we choose these interpolation points to be the rightmost $r$ poles of $H$.
For state-space systems, such as the ISS 1R example,
these are easily computed from the eigenvalues of $\ma A$
using an iterative Krylov solver like ARPACK~\cite{LSY98}.\linelabel{line:arpack}
For the delay example, we computed the poles of $H$
by maximizing $|H(z)|$ starting from pole estimates. 

For IRKA and TF-IRKA these $r$ poles provide sufficient
data to construct rational interpolants
using both transfer function evaluations $H(\mu_j)$ and derivatives $H'(\mu_j)$.
However this same initialization yields an underdetermined 
rational approximation problem in Projected $\Htwo$ 
as our approach only uses evaluations.
Hence, our Projected $\Htwo$ approach constructs
lower dimensional reduced order models until sufficient data 
has been accumulated to yield an overdetermined problem.
Note that for the $r=2$, we do not have sufficient data for an even 
reduced order model and hence use the four rightmost poles of $H$
in this case only.

QuadVF does not need initial shifts,
instead taking fixed samples of $H$ along the imaginary axis.
However an initialization is still needed for the vector fitting iteration;
here we use the poles from AAA as in our Projected $\Htwo$ approach.
We choose scaling parameter $L=10$
and set the number of quadrature points to approximately equal the number of
evaluations of $H$ used by our Projected $\Htwo$ approach at $r=50$. 
With $L=10$ and $N=100$ the quadrature nodes with positive imaginary part
are in $[7.8\times 10^{-2}, 6.4\times 10^2]\iu$
which covers the active region of the Bode plot.\linelabel{line:Lchoice}

\subsubsection{Termination Condition}
As each algorithm is derived from different principles,
each has different natural termination conditions.
To provide the same termination condition for each algorithm, 
we terminate when the difference between successive iterates $H_r^\ell$ and $H_r^{\ell+1}$ is sufficiently small;
namely,  
\begin{equation}
	\| H_r^\ell - H_r^{\ell+1}\|_\Htwo \le 10^{-9}.
\end{equation}
This small tolerance is necessary to avoid premature termination
and is easily computed in $\order(r^3)$ operations via~\cref{eq:H2_gramian}.\linelabel{line:xtol}

\input{fig_space_station}
\input{fig_delay}

\subsection{Discussion}
Figures~\ref{fig:iss} and \ref{fig:delay}
illustrate the performance of several model reduction algorithms
on the space station and delay models.

\subsubsection{\PHtwo\ Performance}
As these experiments illustrate,
our Projected $\Htwo$ algorithm (\PHtwo)  often converges to the similar or a better local minimizer
when compared to IRKA, TF-IRKA, and QuadVF.
Moreover, our algorithm always uses fewer evaluations of the transfer function
than IRKA and TF-IRKA, often by an order of magnitude.
This is not an artifact of the tight termination criteria.
As the convergence histories in both examples indicate,
the \PHtwo\  converges before both IRKA and TF-IRKA
have taken their third step.
This improved performance is the result of \PHtwo\ 
recycling evaluations of the transfer function between steps
whereas IRKA and TF-IRKA must discard previous transfer function evaluations.

\subsubsection{Non-monotonicity in Degree}
We expect that $\|H - H_r\|_\Htwo$ should decrease monotonically with increasing degree $r$
as $\set R_r^+(\R) \subset \set R_{r+1}^+(\R)$.
This is not necessarily true in our experiments
as we can only recover local minimizers---not the global minimizer---%
since the $\Htwo$-model reduction problem is nonconvex.

\subsubsection{Non-monotonicity of Iterates}
These examples also illustrate 
that the error $\| H - \widehat H_r^\ell\|_\Htwo$
does not monotonically decrease with iteration $\ell$ 
(note evaluations of $H$ effectively count iterations).
In \PHtwo\ this is unsurprising.
Each $\widehat H_r^\ell$ is locally optimal with respect to 
the projection $P(\ve \mu^\ell)$ of the $\Htwo$-norm, not the full $\Htwo$-norm.\linelabel{line:projection_err}
In early iterations, the projected norm is a poor approximation
of the full norm, leading to inaccurate approximations.
This is particularly evident in the delay example.

\subsubsection{Oscillation in QuadVF\label{sec:results:discussion:oscillation}}
Both examples show oscillation in
the QuadVF $\Htwo$ error as we increase the number of quadrature points for a fixed degree
as seen in the bottom plot.
This is a result of the quadrature rule used to approximate the $\Htwo$ norm.
When these quadrature nodes are nearby the peaks in the Bode plot
we obtain a better fit than when they are far away.
Consequently with a large number of quadrature points
we are able to find a good approximation;
however, prior to that
the approximation quality will depend on 
the location of the quadrature nodes~\cref{eq:H2_bcc}.

\subsubsection{Spurious Poles}
Finally, we note that \PHtwo\ sometimes fails to eliminate spurious poles
leading to a larger mismatch than expected.
This occurs in the delay example at $r=20$ and $r=30$
and illustrates that while the heuristic for removing spurious poles
described in~\cref{sec:outer:algorithm:spurious} often succeeds, 
it is not foolproof.


%% file: fig_space_station.tex
\begin{figure}
\pgfplotstableread{data/fig_iss_ph2.dat}{\phtwo}
\pgfplotstableread{data/fig_iss_irka.dat}{\irka}
\pgfplotstableread{data/fig_iss_tfirka.dat}{\tfirka}
\centering
\begin{tikzpicture}
\begin{groupplot}[
		group style = {group size = 1 by 4, vertical sep = 4em},
		width = 0.975\textwidth,
		height = 0.36\textwidth,
		title style={yshift = -7pt,},
	]
	\nextgroupplot[
			xmin = 1e-1, xmax = 1e3,
			xlabel = {frequency, $\omega$},
			ylabel = {$|H(i\omega)|$},
			ymax = 1,
			ymin = 1e-8,
			ytickten = {-8,-7,...,0},
			legend style = {
				at = {(1, 1)},
				anchor = north east,
				draw = none,
				fill = none,
				xshift = -5pt,
				yshift = -5pt,
			},
			title = {International Space Station, 1R Component, 1st input / 1st output},
			xmode = log, ymode = log,
		]
		\addplot[black, thick] 
			table [x=z, y=Hz] {data/fig_iss_ph2_bode_28.dat}
			node[pos=0, anchor=south west, yshift=0pt] {$H$};
		
		\addplot[colorbrewerA3, thick,densely dashed]
			 table [x=z, y=diff]{data/fig_iss_tfirka_bode_28.dat}
			node[pos=0, anchor=south west, yshift=-2pt] {TF-IRKA};
		\addplot[colorbrewerA2, thick, densely dashdotted]
			 table [x=z, y=diff]{data/fig_iss_irka_bode_28.dat}
			node[pos=0, anchor=south west, yshift=-2pt] {IRKA};
		\addplot[colorbrewerA1, thick, densely dotted]
			table [x=z, y=diff]{data/fig_iss_ph2_bode_28.dat}
			node[pos=0.00, anchor=west, yshift=-2pt, xshift=30pt] {Projected $\Htwo$};

	\nextgroupplot[
			xtick = {2,4,...,50},
			xmin = 1, xmax = 51.5,
			ymin = 1e-4, ymax = 1,
			ytickten = {-4,-3,...,0},
			xlabel = {ROM dimension $r$},
			ylabel = {relative $\Htwo$ error},
			legend style = {
				at = {(1, 1)},
				anchor = north east,
				draw = none,
				fill = none,
				xshift = -2pt,
				yshift = -2pt,
			},
			legend columns = 1,
			legend cell align = left,
			ymode = log,
			title = Reduced Order Model Error at Termination,
			clip = false,
		]
		\addplot[colorbrewerA1, mark=*, only marks, thick, mark size = 1.2] 
				table [x=rom_dim, y=rel_err] {data/fig_iss_ph2.dat};
		\addlegendentry{Projected $\Htwo$};
		\addplot[colorbrewerA2, mark=Mercedes star flipped, only marks, thick, mark size = 3] 
			table [x=rom_dim, y=rel_err] {data/fig_iss_irka.dat};
		\addlegendentry{IRKA};
		\addplot[colorbrewerA3, mark=Mercedes star, only marks, thick, mark size = 3] 
			table [x=rom_dim, y=rel_err] {data/fig_iss_tfirka.dat};
		\addlegendentry{TF-IRKA};

		\addplot[colorbrewerA4, mark=o, only marks, thick, mark size = 2] 
			table [x=rom_dim, y=rel_err] {data/fig_iss_quadvf.dat};
		\addlegendentry{QuadVF};

		\addplot[colorbrewerA1, mark=*, only marks, thick, mark size = 1.2] 
				table [x=rom_dim, y=rel_err] {data/fig_iss_ph2.dat};

	\nextgroupplot[
		xmin = 1, xmax = 51.5,
		ymin = 1, ymax = 1e4,
		xtick = {2,4,...,50},
		xlabel = {ROM dimension $r$},
		ymode = log,
		ytickten = {0,1,2,3,4},
		yticklabels = {$^{\phantom{-}}10^0$,
			$^{\phantom{-}}10^1$,
			$^{\phantom{-}}10^2$,
			$^{\phantom{-}}10^3$,
			$^{\phantom{-}}10^4$,
		}, 
		ylabel = {linear solves},
		title = Cost to Satisfy Termination Criteria, 
	]
		\addplot[colorbrewerA2, mark=Mercedes star flipped, only marks, thick, mark size = 3] 
			table [x=rom_dim, y=fom_evals] {data/fig_iss_irka.dat};

		\addplot[colorbrewerA3, mark=Mercedes star, only marks, thick, mark size = 3] 
			table [x=rom_dim, y=fom_evals] {data/fig_iss_tfirka.dat};
		
		\addplot[colorbrewerA4, mark=o, only marks, thick, mark size = 2] 
			table [x=rom_dim, y=fom_evals] {data/fig_iss_quadvf.dat};

		\addplot[colorbrewerA1, mark=*, only marks, thick, mark size = 1.2] 
				table [x=rom_dim, y=fom_evals] {data/fig_iss_ph2.dat};

	\nextgroupplot[
			xmin = 0, xmax = 5e2,
			ymin = 1e-3, ymax = 1,
			ymode = log,
			ytickten = {-3,-2,...,0},
			xlabel = {evaluations of $H(z)$, $H'(z)$, and linear solves},
			ylabel = {relative $\Htwo$ error},
			title = {Convergence History $r=28$},
		]
		
		\addplot[colorbrewerA4, const plot mark left, thick] 
			table [x=N, y=rel_err] {data/fig_iss_quadvf_28_L10.dat}
			node [pos=0.5, anchor = north east, yshift=-2pt] {{QuadVF $L=10$}};

		\addplot[colorbrewerA2, const plot mark left, thick] 
			table [x=fom_evals, y=rel_err] {data/fig_iss_irka_28.dat}
			node [pos=1] {$\bullet$}
			node [pos=1, anchor = south east, yshift = 0pt] {IRKA};

		\addplot[colorbrewerA3, const plot mark left, thick] 
			table [x=fom_evals, y=rel_err] {data/fig_iss_tfirka_28.dat}
			node [pos=1] {$\bullet$}
			node [pos=1, anchor = south east] {TF-IRKA};
		
		\addplot[colorbrewerA1, const plot mark left, thick, ] 
			table [x=fom_evals, y=rel_err] {data/fig_iss_ph2_28.dat}
			node [pos=1] {$\bullet$}
			node[pos=1, anchor = west, yshift=-2pt,xshift=0pt] {Projected $\Htwo$};
\end{groupplot}
\end{tikzpicture}
\vspace{-1em}
\caption{A comparison of model reduction techniques applied the 1R component of the International Space Station 
described in~\cite[subsec.~2.11]{CD02}.
The top plot shows the modulus of the transfer function along the imaginary axis,
with the broken lines showing the value of the error system $H- H_r$ for different techniques at $r=28$.
The second plot shows the relative $\Htwo$ error for each method for a variety of reduced order model dimensions
and the table below shows the number of linear solves, or equivalently, evaluations of $H(z)$ and $H'(z)$ required.
The bottom plot shows the convergence history of each of these methods 
along with a comparison to QuadVF.
}
\label{fig:iss}
\end{figure}
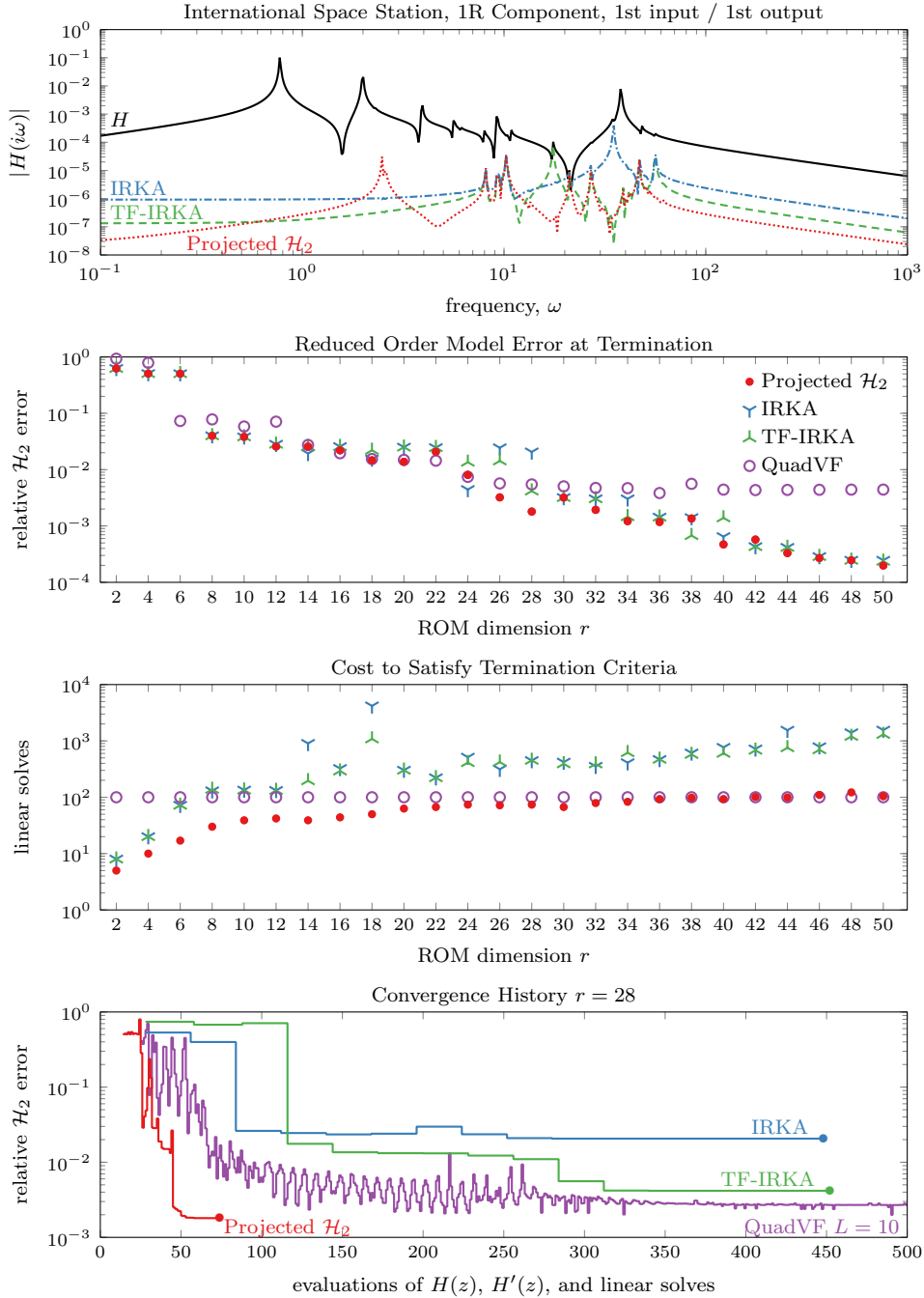

%% file: fig_delay.tex
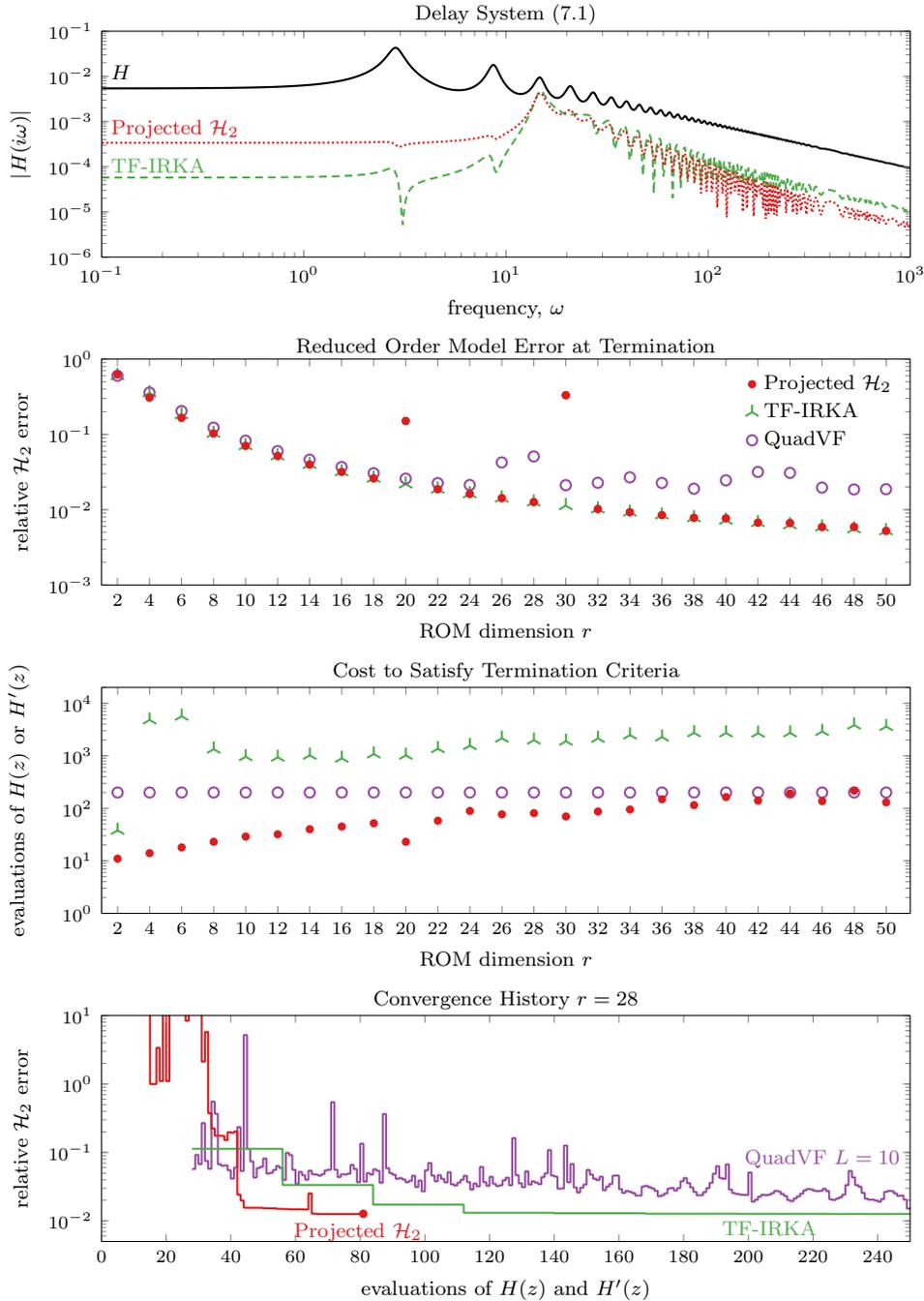
\begin{figure}
\pgfplotstableread{data/fig_delay_ph2.dat}{\phtwo}
\pgfplotstableread{data/fig_delay_tfirka.dat}{\tfirka}
\centering
\begin{tikzpicture}
\begin{groupplot}[
		group style = {group size = 1 by 4, vertical sep = 4em},
		width = 0.975\textwidth,
		height = 0.36\textwidth,
		title style={yshift = -7pt,},
	]
	\nextgroupplot[
			xmin = 1e-1, xmax = 1e3,
			xlabel = {frequency, $\omega$},
			ylabel = {$|H(i\omega)|$},
			ymax = 1e-1,
			ymin = 1e-6,
			ytickten = {-8,-7,...,0},
			legend style = {
				at = {(1, 1)},
				anchor = north east,
				draw = none,
				fill = none,
				xshift = -5pt,
				yshift = -5pt,
			},
			title = {Delay System~\cref{eq:delay}},
			xmode = log, ymode = log,
		]
		\addplot[black, thick] 
			table [x=z, y=Hz] {data/fig_delay_ph2_bode_08.dat}
			node[pos=0, anchor=south west, yshift=0pt] {$H$};
		
		\addplot[colorbrewerA3, thick,densely dashed]
			 table [x=z, y=diff]{data/fig_delay_tfirka_bode_06.dat}
			node[pos=0, anchor=south west, yshift=-2pt] {TF-IRKA};
		\addplot[colorbrewerA1, thick, densely dotted]
			table [x=z, y=diff]{data/fig_delay_ph2_bode_06.dat}
			node[pos=0, anchor=south west, yshift=-2pt, xshift=0pt] {Projected $\Htwo$};
	
	\nextgroupplot[
			xtick = {2,4,...,50},
			xmin = 1, xmax = 51.5,
			ymin = 1e-3, ymax = 1,
			ytickten = {-4,-3,...,0},
			xlabel = {ROM dimension $r$},
			ylabel = {relative $\Htwo$ error},
			legend style = {
				at = {(1, 1)},
				anchor = north east,
				draw = none,
				fill = none,
				xshift = -2pt,
				yshift = -2pt,
			},
			legend columns = 1,
			legend cell align = left,
			ymode = log,
			title = Reduced Order Model Error at Termination,
			clip = true,
		]
		\addplot[colorbrewerA1, mark=*, only marks, thick, mark size = 1.2] 
				table [x=rom_dim, y=rel_err] {data/fig_delay_ph2.dat};
		\addlegendentry{Projected $\Htwo$};

		\addplot[colorbrewerA3, mark=Mercedes star, only marks, thick, mark size = 3] 
			table [x=rom_dim, y=rel_err] {data/fig_delay_tfirka.dat};
		\addlegendentry{TF-IRKA};

		\addplot[colorbrewerA4, mark=o, only marks, thick, mark size = 2] 
			table [x=rom_dim, y=rel_err] {data/fig_delay_quadvf.dat};
		\addlegendentry{QuadVF};

		\addplot[colorbrewerA1, mark=*, only marks, thick, mark size = 1.2] 
				table [x=rom_dim, y=rel_err] {data/fig_delay_ph2.dat};

	\nextgroupplot[
		xmin = 1, xmax = 51.5,
		ymin = 1, ymax = 2e4,
		xtick = {2,4,...,50},
		xlabel = {ROM dimension $r$},
		ymode = log,
		ytickten = {0,1,2,3,4,5},
		yticklabels = {$^{\phantom{-}}10^0$,
			$^{\phantom{-}}10^1$,
			$^{\phantom{-}}10^2$,
			$^{\phantom{-}}10^3$,
			$^{\phantom{-}}10^4$,
			$^{\phantom{-}}10^5$,
		}, 
		ylabel = {evaluations of $H(z)$ or $H'(z)$},
		title = Cost to Satisfy Termination Criteria, 
	]
		\addplot[colorbrewerA3, mark=Mercedes star, only marks, thick, mark size = 3] 
			table [x=rom_dim, y=fom_evals] {data/fig_delay_tfirka.dat};
		
		\addplot[colorbrewerA4, mark=o, only marks, thick, mark size = 2] 
			table [x=rom_dim, y=fom_evals] {data/fig_delay_quadvf.dat};

		\addplot[colorbrewerA1, mark=*, only marks, thick, mark size = 1.2] 
				table [x=rom_dim, y=fom_evals] {data/fig_delay_ph2.dat};

	\nextgroupplot[
			xmin = 0, xmax = 2.5e2,
			ymin = 5e-3, ymax = 10e0,
			ymode = log,
			ytickten = {-3,-2,...,0,1},
			xlabel = {evaluations of $H(z)$ and $H'(z)$},
			ylabel = {relative $\Htwo$ error},
			title = {Convergence History $r=28$},
		]
		
		\addplot[colorbrewerA4, const plot mark left, thick] 
			table [x=fom_evals, y=rel_err] {data/fig_delay_quadvf_hist_28.dat}
			node [pos=0.5, anchor = south east, yshift=12pt] {{QuadVF $L=10$}};

		\addplot[colorbrewerA3, const plot mark left, thick] 
			table [x=fom_evals, y=rel_err] {data/fig_delay_tfirka_hist_28.dat}
			node [pos=1] {$\bullet$}
			node [pos=0.1, anchor = north east, yshift=1pt] {TF-IRKA};

		\addplot[colorbrewerA1, const plot mark left, thick, ] 
			table [x=fom_evals, y=rel_err] {data/fig_delay_ph2_hist_28.dat}
			node [pos=1] {$\bullet$}
			node[pos=1, anchor = north, yshift=2pt,xshift=0pt] {Projected $\Htwo$};
\end{groupplot}
\end{tikzpicture}
\vspace{-1em}
\caption{A comparison of model reduction techniques applied the 
delay system given in~\cref{eq:delay}.
The top plot shows the modulus of the transfer function along the imaginary axis,
with the broken lines showing the value of the error system $H- H_r$ for different techniques at $r=6$.
The second plot shows the relative $\Htwo$ error for each method for a variety of reduced order model dimensions
and the table below shows the number of evaluations of $H(z)$ and $H'(z)$ required.
The bottom plot shows the convergence history of each of these methods.
Here we approximate $\Htwo$-norm approximately using a Boyd/Clenshaw-Curtis quadrature rule~\cref{eq:H2_bcc}
with $10^4$ quadrature points.
}
\label{fig:delay}
\end{figure}

%% file: discussion.tex
\section{Conclusion}
We have developed the \emph{Projected $\Htwo$} approach
for $\Htwo$-optimal model reduction problem
by applying the projected nonlinear least squares framework to this problem.
This allows the $\Htwo$ problem to be converted
into a sequence of finite-dimensional rational approximation problems.
Although solving these rational approximation problems is more challenging
and computationally expensive than constructing rational interpolants as in IRKA and TF-IRKA,
this cost is justified by requiring far fewer expensive evaluations of the full order model.

%% file: appendix.tex
\section{\label{sec:appendix}\nopunct} 
We now provide the proofs for \cref{thm:angle} and \cref{thm:conv}.
A key component in both proofs is the following lemma.
\begin{lemma}
	\label{lem:quadprog}
	Let  $f:\C^m \times \C^n \to \R$ be defined as 
	\begin{equation}
		f(\ve x, \ve y) \coloneqq
		\begin{bmatrix}
		\ve x	\\ \ve y
		\end{bmatrix}^* 
		\begin{bmatrix}
			\ma A & \ma B \\ \ma B^* & \ma C 
		\end{bmatrix}
		\begin{bmatrix}
			\ve x \\ \ve y 
		\end{bmatrix}
		\quad \text{where} \quad
		\ma A = \ma A^*, \ma C = \ma C^*,
	\end{equation}
	and $\ma A$ is positive definite,
	then 
	\begin{equation}\label{eq:quadprog}
		\min_{\ve x\in \C^m} f(\ve x, \ve y) =
			\ve y^*[\ma C - \ma B^*\ma A^{-1} \ma B] \ve y.
	\end{equation}
\end{lemma}
\begin{proof}
	We use Wirtinger calculus (see, e.g.,~\cite[App.~A]{SS10})
	that treats $\ve x \in \C^m$ and its conjugate $\ove x$
	as independent variables whose partial derivatives are
	\begin{align}
		\label{eq:wirtinger}
		\frac{\partial \ve x}{\partial \ve x} &= \ma I, &
		\frac{\partial \ve x}{\partial \ove x} &= \ma 0, &
		\frac{\partial \ove x}{\partial \ve x} &= \ma 0, &
		\frac{\partial \ove x}{\partial \ove x} &= \ma I.
	\end{align}
	Hence the first derivatives of $f$ with respect to $\ve x$ and $\ove x$ are :
	\begin{align}
		\frac{\partial f}{\partial \ve x} &=  \ve x^*\ma A + \ve y^* \ma B^* &
		\frac{\partial f}{\partial \ove x} &= \ve x^\trans \ma A^* + \ve y^\trans \ma B^\trans. 
	\end{align}
	Setting $\ve x$ to be $\hve x = -\ma A^{-1}\ma B\ve y$
	makes both derivatives zero;
	hence $\hve x$ satisfies the first order necessary conditions.
	As $f$ is convex in $\ve x$ as $\ma A$ is positive definite, 
	the local minimizer $\hve x$ is the global minimizer.
\end{proof}

\subsection{Proof of~\cref{thm:angle}}
We split this proof into two lemmas that bound
terms that later appear in the proof of \cref{thm:angle}.

\begin{lemma}\label{lem:p1}
	Suppose $\lambda \in \C_-$ and $\mu_1 \in \C_+$
	where $|\lambda + \conj\mu_1| \le \epsilon$.
	Then there exists a constant $C_1>0$ independent of $\epsilon$ such that
	\begin{equation}
		\min_{x_1 \in \C} 
		\| v[-\conj{\lambda}] z_1 + v[\mu_1] x_1\|_\Htwo
		\le |z_1| C_1 \epsilon.
	\end{equation}
\end{lemma}
\begin{proof}
Examining the squared objective
\begin{equation}
	\| v[-\conj{\lambda}] z_1 + v[\mu_1] x_1\|_\Htwo^2
	=
	\begin{bmatrix}
		x_1 \\ z_1 
	\end{bmatrix}^*
	\begin{bmatrix}
		(\mu_1 + \conj\mu_1)^{-1} & (\mu_1 - \lambda)^{-1} \\
		(-\conj \lambda + \conj\mu_1 )^{-1} & (-\conj\lambda -\lambda)^{-1} 
	\end{bmatrix}
	\begin{bmatrix}
		x_1 \\ z_1 
	\end{bmatrix}, 
\end{equation}
we note that we can apply \cref{lem:quadprog} 
as $(\mu_1 +\conj\mu_1)^{-1}$ is positive.
After simplification, 
\begin{equation}\label{eq:bound_p1}
	\begin{split}
	\min_{x_1 \in \C} 
		\| v[-\conj{\lambda}] z_1 + v[\mu_1] x_1\|_\Htwo^2
		&= |z_1|^2 \frac{ |\mu_1 + \conj\lambda|^2}{2|\mu_1 - \lambda|^2 \Re[-\lambda]}.
	\end{split} 
\end{equation}
Setting $\mu_1 = -\conj\lambda + \delta$ 
where $|\delta| \le \epsilon$ 
and as $\Re \mu_1 > 0$,  $\Re \delta > \Re \lambda$, then
\begin{equation}\label{eq:mu_lam_bnd}
	|\mu_1 -\lambda| = |-\conj\lambda + \delta -\lambda| = |-2\Re \lambda + \delta| 
	\ge | \Re\lambda|.
\end{equation}
Using this bound in the numerator yields the result:
\begin{equation}
	\min_{x_1 \in \C} 
		\| v[-\conj{\lambda}] z_1 + v[\mu_1] x_1\|_\Htwo^2
	\le |z_1|^2 \frac{ |\mu_1 + \conj\lambda|^2}{2|\Re\lambda|^3}
	\le |z_1|^2 C_1^2 \epsilon^2.
\end{equation}
\end{proof}

\begin{lemma}\label{lem:p2}
	Suppose $\lambda \in \C_-$ and $\mu_2,\mu_3 \in \C_+$
	where $|\lambda + \conj\mu_2| \le \epsilon$, 
	$|\lambda + \conj\mu_3| \le \epsilon$, and $\mu_2 \ne \mu_3$.
	Then there exists a constant $C_2> 0$ independent of $\epsilon$ such that
	\begin{equation}
	\min_{x_2, x_3 \in \C}
	\|
		v[-\conj{\lambda}]' z_{2}
		+ v[\mu_{2}] x_{2} 
		+ v[\mu_{3}] x_{3}
	\|_\Htwo \le
		|z_2| C_2 \epsilon.
	\end{equation}
\end{lemma}
\begin{proof}
Inserting the additive identity $0 = v[\mu_2]' - v[\mu_2]'$ into the objective
and applying the triangle inequality
\begin{multline}
	\label{eq:bound_p2}
	\|
		v[-\conj{\lambda}]' z_{2}
		+ v[\mu_{2}] x_{2} 
		+ v[\mu_{3}] x_{3}
	\|_\Htwo \\ = 
	|z_2| \| v[-\conj\lambda]' -  v[\mu_2]'\|_\Htwo
	+
	\|
		v[\mu_2]' z_{2}
		+ v[\mu_{2}] x_{2} 
		+ v[\mu_{3}] x_{3}
	\|_\Htwo.
\end{multline}
The first term above is
\begin{equation}
	\| v[-\conj\lambda]' -  v[\mu_2]'\|_\Htwo^2
	= 
	\frac{2}{(-\conj\lambda - \lambda)^{3}}
	-\frac{2}{(-\conj\lambda + \conj\mu_2)^{3}}
	-\frac{2}{(\mu_2 -\lambda)^{3}}
	+\frac{2}{(\mu_2 + \conj\mu_2)^{3}}.
\end{equation}
After some simplification and noting $|\lambda - \mu_2| \ge |\Re\lambda|$ following~\cref{eq:mu_lam_bnd},
\begin{equation}
	\label{eq:bound_p2a}
	\| v[-\conj\lambda]' -  v[\mu_2]'\|_\Htwo^2
	 = 2\frac{ 
			|\mu_2 + \conj\lambda|^2 \, p(\mu_2, \lambda, \conj{\mu_2}, \conj{\lambda})
		}{
		|\lambda -\mu_2|^6 \Re[2\mu_2]^3 \Re[-2\lambda]^3
		}
		\le 
	  \frac{ 
			|\mu_2 + \conj\lambda|^2 \, p(\mu_2, \lambda, \conj{\mu_2}, \conj{\lambda})
		}{
		2^5|\Re \lambda|^9 \Re[\mu_2]^3 
		}
		\le C_3^2 \epsilon^2
\end{equation}
where $p(\mu_2, \lambda, \conj{\mu_2}, \conj{\lambda})$ is a degree-7 polynomial
and $C_3> 0$ captures the terms independent of $\epsilon$.
Next, we bound the second term of~\cref{eq:bound_p2};
expanding this term
\begin{multline}
	\|
		v[\mu_{2}] x_{2} 
		+ v[\mu_{3}] x_{3}
		+ v[\mu_2]' z_{2}
	\|_\Htwo^2 \\
	\setlength\arraycolsep{3pt}
	=\begin{bmatrix}
		x_2 \\ x_3 \\ z_2
	\end{bmatrix}^* \!
	\begin{bmatrix}
		(\mu_2 + \conj\mu_2)^{-1} & (\mu_2 + \conj\mu_3)^{-1} & - (\mu_2 + \conj\mu_2)^{-2} \\
		(\mu_3 + \conj\mu_2)^{-1} & (\mu_3 + \conj\mu_3)^{-1} & - (\mu_3 + \conj\mu_2)^{-2} \\
		-(\mu_2 + \conj\mu_2)^{-2} & -(\mu_2 +\conj\mu_3)^{-2} & 2(\conj\mu_2 +\mu_2)^{-3}
	\end{bmatrix} \!
	\begin{bmatrix}
		x_2 \\ x_3 \\ z_2
	\end{bmatrix}.
\end{multline}
As $\mu_2$ and $\mu_3$ are distinct, the upper left $2\times 2$ block is positive definite
and we may apply \cref{lem:quadprog} which yields, after some simplification,
\begin{multline}\label{eq:bound_p2b}	
	\min_{x_2, x_3 \in \C}\|v[-\conj{\lambda}]' z_2 + v[\mu_2] x_2 + v[\mu_3] x_3 \|_\Htwo^2  
	= \frac{ |z_2|^2 |\mu_2 - \mu_3|^2}{\Re[2\mu_2]^3|\mu_2 + \conj\mu_3|^2} 
	\le C_4^2 |z_2|^2 \epsilon^2,
\end{multline}
where $C_4^2$ captures the terms that are independent of $\epsilon$.

Combining \cref{eq:bound_p2a} and \cref{eq:bound_p2b}, we obtain the bound
\begin{equation}
	\min_{x_2, x_3 \in \C}
	\|
		v[-\conj{\lambda}]' z_{2}
		+ v[\mu_{2}] x_{2} 
		+ v[\mu_{3}] x_{3}
	\|_\Htwo \le
		|z_2|\epsilon \sqrt{ C_3^2 + C_4^2}.
\end{equation}
\end{proof}

\begin{proof}[Proof of \cref{thm:angle}]
We begin by rewriting the sine of the subspace angle in terms of 
the unitary basis operator $B_{\set T(\ve\mu)}$ using~\cite[eq.~(13)]{BG73}
\begin{equation}\label{eq:sin_max_norm}
	\begin{split}
	\sin \phi_{\max} (\set T(\ve\lambda), \set V(\ve\mu)) 
	&=\| (I - P(\ve\mu))B_{\set T(\ve\lambda)}B_{\set T(\ve \lambda)}^*\|_\Htwo \\
	&= \max_{\ve y \in \C^{2r}, \|\ve y\|_2=1} 
			\| (I - P(\ve\mu)) B_{\set T(\ve\lambda)} \ve y\|_\Htwo.
	\end{split}
\end{equation}
As $I-P(\ve\mu)$ is an orthogonal projector onto 
the complement of $\set V(\ve\mu)$,
its projection satisfies the closest point property.
This permits an equivalent restatement as
\begin{align}
	\sin \phi_{\max} (\set T(\ve\lambda), \set V(\ve\mu)) 
	&= \max_{\ve y \in \C^{2r}, \|\ve y\|_2=1} 
		\min_{\ve x \in \C^n } 
		\| B_{\set T(\ve\lambda)} \ve y + V(\ve\mu) \ve x\|_\Htwo.
\end{align}
Writing this in terms of the non-orthogonal basis vectors for $\set T(\ve\lambda)$ and $\set V(\ve\mu)$, 
\begin{align}\label{eq:sin_maxmin}
	\sin \phi_{\max} (\set T(\ve\lambda), \set V(\ve\mu)) 
	&\! =\! \! \max_{\substack{\ve z \in \C^{2r} \\ \ve z \ne 0}} 
		\min_{\ve x \in \C^n } 
		\frac{\left\| 
			\setlength\arraycolsep{3pt}
			\begin{bmatrix} V(-\conj{\ve\lambda}) & V'(-\conj{\ve\lambda}) \end{bmatrix} \ve z
			 + V(\ve\mu) \ve x\right
		\|_\Htwo 
	}{\|\hma M(\ve\lambda) \ve z\|_2}.
\end{align}
We now bound the numerator by making a non-optimal choice of $\ve x$.
Denoting the entires of $\ve x$ associated with $\mu_{k,t}$ as $x_{k,t}$
and setting the remainder to zero, we have
\begin{multline}
	\min_{\ve x \in \C^n} 
		\left\| 
			\begin{bmatrix} V(-\conj{\ve\lambda}) & V'(-\conj{\ve\lambda}) \end{bmatrix} \ve z
	 		-V(\ve\mu) \ve x 
		\right\|_\Htwo 
	\\
	\le \! \min_{\substack{ 
			\lbrace x_{k,1} \rbrace_{k=1}^m \subset \C \\ 
			\lbrace x_{k,2} \rbrace_{k=1}^m \subset \C \\ 
			\lbrace x_{k,3} \rbrace_{k=1}^m \subset \C 
		}}
	\left\| \sum_{k=1}^m 
			v[-\conj{\lambda}_k] z_{k,1}
			\!+\!v[-\conj{\lambda}_k]' z_{k,2}
			\!+\! v[\mu_{k,1}] x_{k,1} 
			\!+\! v[\mu_{k,2}] x_{k,2} 
			\!+\! v[\mu_{k,3}] x_{k,3}\right\|_\Htwo\!\!\!.\!\!\!
\end{multline}
Applying the triangle inequality 
by grouping $v[-\conj\lambda_k]$ with $v[\mu_{k,1}]$ 
and $v'[-\conj\lambda_k]$ with $v[\mu_{k,2}]$ and $v[\mu_{k,3}]$
yields the upper bound  
\begin{multline}\label{eq:bound_split}
	\setlength\arraycolsep{2pt}
	\min_{\ve x \in \C^n} 
		\left\| 
			\begin{bmatrix} V(-\conj{\ve\lambda}) & V'(-\conj{\ve\lambda}) \end{bmatrix} \ve z
	 		-V(\ve\mu) \ve x 
		\right\|_\Htwo 
	\le \\ 
	\sum_{k=1}^m 
	\min_{\substack{x_{k,1}\in \C }}
	\| 
		v[-\conj{\lambda}_k] z_{k,1}
			+ v[\mu_{k,1}] x_{k,1}
	\|_\Htwo \\
	 +
	\sum_{k=1}^m
	\min_{\substack{x_{k,2}\in \C \\ x_{k,3} \in \C}}
	\|
		v[-\conj{\lambda}_k]' z_{k,2}
		+ v[\mu_{k,2}] x_{k,2} 
		+ v[\mu_{k,3}] x_{k,3}
	\|_\Htwo.
\end{multline}
In each term, the minimization has been pulled inside the sum
as each of $x_{k,1}$ and $x_{k,2}, x_{k,3}$ appear in only one term.
Invoking \cref{lem:p1} and \cref{lem:p2}
we have
\begin{equation}\label{eq:sum_bnd}
	\min_{\ve x \in \C^n}
		\setlength\arraycolsep{3pt} 
		\left\| 
			\begin{bmatrix} V(-\conj{\ve\lambda}) & V'(-\conj{\ve\lambda}) \end{bmatrix} \ve z
	 		-V(\ve\mu) \ve x 
		\right\|_\Htwo  
	\le \epsilon \sum_{k=1}^m C_{k,1} |z_{k,1}| + C_{k,2} |z_{k,2}| .
\end{equation}
The matrix $\hma M(\ve\lambda)$ is invertible as the entries of $\ve\lambda$ are distinct
which provides the lower bound
$\|\hma M(\ve\lambda) \ve z\|_2 \ge \sigma_{\min}(\hma M(\ve\lambda)) \|\ve z\|_\infty$.
Returning to~\cref{eq:sin_maxmin} with this lower bound and the upper bound in~\cref{eq:sum_bnd},
\begin{equation}
	\sin\phi_{\max}(\set T(\ve\lambda), \set V(\ve\mu))
		\le
		\frac{\epsilon \sum_{k=1}^m C_{k,1} |z_{k,1}| \!+\! C_{k,2} |z_{k,2}|}{
			\sigma_{\min}(\hma M(\ve\lambda)) \|\ve z\|_\infty}
		\le \epsilon \frac{\sum_{k=1}^m C_{k,1} \!+\! C_{k,2}}{\sigma_{\min}(\ma M(\hve \lambda))},
\end{equation}
where the last step follows from $|z_{k,1}|/\|\ve z\|_\infty\le 1$.
\end{proof}

\subsection{Proof of~\cref{thm:conv}}
We begin by establishing a subspace angle result analogous 
to those in \cref{lem:p1} and \cref{lem:p2}.

\begin{lemma}\label{lem:Qangle_full}
	Suppose $\ve\lambda, \ve \xi \in \C^r$ where
	$|\lambda_j - \xi_j| \le \epsilon$.
	Then there exists a constant $C>0$ independent of $\epsilon$ such that
	\begin{equation}
		\sin \phi_{\max} ( \set T(\ve \lambda), \set T(\ve \xi)) \le C \epsilon. 
	\end{equation}
\end{lemma}
\begin{proof}
As in~\cref{eq:sin_maxmin},
this subspace angle is 
\begin{equation}
	\sin \phi_{\max}( \set T(\ve\lambda), \set T(\ve \xi))
		= 
	\max_{\substack{\ve z \in \C^{2r} \\ \ve z \ne 0}}
	\min_{\ve x\in \C^{2r}}
		\frac{
			\|
				\begin{bmatrix} V(-\ove\lambda) & V'(-\ove\lambda )\end{bmatrix} \ve z + 
				\begin{bmatrix} V(-\ove \xi) & V'(-\ove\xi)\end{bmatrix} \ve x
		 	\|_\Htwo
		}{
			\|\hma M(\ve\lambda)\ve z\|_2
		}.
\end{equation}
Following~\cref{eq:bound_split}, we make a suboptimal choice of $\ve x$, 
pairing those associated with each $\lambda_j$ and $\xi_j$
and indexing the corresponding entries of $\ve x$ and $\ve z$ by 
$x_{j,1}, x_{j,2}$ and $z_{j,1}, z_{j,2}$:
\begin{multline}
	\min_{\ve x\in \C^{2r}}
		\|
			\begin{bmatrix} V(-\ove\lambda) & V'(-\ove\lambda )\end{bmatrix} \ve z + 
			\begin{bmatrix} V(-\ove \xi) & V'(-\ove\xi)\end{bmatrix} \ve x
		\|_\Htwo
	\le \\ 
	\sum_{j=1}^r \min_{x_{j,1},x_{j,2}\in \C} 
		\|
			v[-\conj \lambda_j] z_{j,1} + v[-\conj\lambda_j]' z_{j,2} + 
			 v[-\conj \xi_j] x_{j,1} + v[-\conj\xi_j]' x_{j,2}
		\|_\Htwo.
\end{multline} 
Examining the objective function
\begin{multline}		
	\setlength\tabcolsep{1pt}
	\|v[-\conj\lambda_j]z_{j,1} + v[-\conj\lambda_j]' z_{j,2}
	+ v[-\conj\xi_j]x_{j,1} + v[-\conj\xi_j]' x_{j,2}
	\|_\Htwo^2\\
	\!=\!
	\begin{bmatrix} z_{j,1} \\ z_{j,2} \\ x_{j,1} \\ x_{j,2} \end{bmatrix}^*\!\!
	\begin{bmatrix}
		(-\conj\lambda_j - \lambda_j)^{-1} & -(-\conj\lambda_j - \lambda_j)^{-2} & (-\conj\lambda_j -\xi_j)^{-1} & -(-\conj\lambda_j - \xi_j)^{-2} \\
		-(-\conj\lambda_j - \lambda_j)^{-2} & 2(-\conj\lambda_j - \lambda_j)^{-3} & -(\conj\lambda_j -\xi_j)^{-2} & 2(-\conj\lambda_j - \xi_j)^{-3} \\
		(-\conj\xi_j - \lambda_j)^{-1} & -(-\conj\xi_j - \lambda_j)^{-2} & (-\conj\xi_j -\xi_j)^{-1} & -(-\conj\xi_j - \xi_j)^{-2} \\
		-(-\conj\xi_j - \lambda_j)^{-2} & 2(-\conj\xi_j - \lambda_j)^{-3} & -(\conj\xi_j -\xi_j)^{-2} & 2(-\conj\xi_j - \xi_j)^{-3} \\
	\end{bmatrix}
	\!\!
	\begin{bmatrix} z_{j,1} \\ z_{j,2} \\ x_{j,1} \\ x_{j,2} \end{bmatrix} \\
	= \begin{bmatrix} \ve z_j \\ \ve x_j \end{bmatrix}^*
	\begin{bmatrix}
		\hma M_j & \ma B_j \\ \ma B_j^* & \ma C_j
	\end{bmatrix}
	\begin{bmatrix} \ve z_j \\ \ve x_j \end{bmatrix}.
\end{multline}
As $\hma M_j$ is positive definite, we may invoke \cref{lem:quadprog} to show
\begin{multline}
	\min_{\ve x_j \in \C^2}
	\|v[-\conj\lambda_j]z_{j,1} + v[-\conj\lambda_j]' z_{j,2}
	+ v[-\conj\xi_j]x_{j,1} + v[-\conj\xi_j]' x_{j,2}
	\|_\Htwo^2 \\
	= \ve z_j^*[\hma M_j - \ma B_j^*\ma C_j^{-1}\ma B_j]\ve z_j.
\end{multline}
With this minimizer we bound the squared subspace angle 
using the singular values
\begin{multline}\label{eq:Qangle_bound}
	\sin^2\phi_{\max} (\set T(\ve\lambda), \set T(\ve \xi))
	\le \max_{\substack{\ve z\in \C^{2r}\\ \ve z\ne 0}}
		\frac{\sum_{j=1}^r \ve z_j^*[\hma M_j - \ma B_j^* \ma C_j^{-1} \ma B_j]\ve z_j}
		{\ve z^*\hma M(\ve\lambda)\ve z} \\
	\le 
		\frac{\sum_{j=1}^r \sigma_{\max}^2(\hma M_j - \ma B_j^* \ma C_j^{-1} \ma B_j)}
		{\sigma_{\min}^2(\hma M(\ve\lambda))}.
\end{multline}
The note the entries of the numerator are all $\order(\epsilon^2)$:
\begin{align}
	[ \hma M_j - \ma B_j^* \ma C_j^{-1} \ma B_j]_{1,1} & = 
		\frac{|\lambda_j -\xi_j|^2 p_1(\lambda_j, \xi_j, \conj\lambda_j, \conj\xi_j)}
		{3(\lambda_j + \conj\lambda_j)^2|\lambda_j+ \conj\xi_j|^4}; \\
	[ \hma M_j - \ma B_j^* \ma C_j^{-1} \ma B_j]_{1,2} & = 
		\frac{|\lambda_j -\xi_j|^2 p_2(\lambda_j, \xi_j, \conj\lambda_j, \conj\xi_j)}
		{(\lambda_j + \conj\lambda_j)^2(\lambda_j + \conj\xi_j)^3(\conj\lambda_j +\xi_j)^2 }; \\
	[ \hma M_j - \ma B_j^* \ma C_j^{-1} \ma B_j ]_{2,2} & = 
		\frac{|\lambda_j -\xi_j|^2 p_3(\lambda_j, \xi_j, \conj\lambda_j, \conj\xi_j)}
		{3(\lambda_j + \conj\lambda_j)^3|\lambda_j+ \conj\xi_j|^6};
\end{align}
where $p_1$, $p_2$, and $p_3$ are all polynomials.
As the denominator of~\cref{eq:Qangle_bound} is independent of $\epsilon$, 
we obtain the desired bound.
\end{proof}

\begin{lemma}\label{lem:poles}
	Suppose $\lbrace \widehat H_r^\ell\rbrace_{\ell=0}^\infty \subset \set R_r^+(\C)$
	has a limit $\widehat H_r^\ell \to \widehat H_r\in \set R_r^+(\C)$.
	If $\widehat H_r$ has $r$ distinct poles, 
	then the poles $\ve \lambda^\ell$ of $\widehat H_r^\ell$
	converge to the poles of $\ve \lambda$ of $\widehat H_r$.
\end{lemma}
\begin{proof}
By assumption, we can write $\widehat{H}_r$ in pole-residue form 
\begin{equation}
	\widehat{H}_r(z) = \sum_{j=1}^r \rho_j (z - \lambda_j)^{-1}
	\quad \text{where} \quad 
	|\rho_j| > 0
\end{equation}
as otherwise $\widehat{H}_r$ would have fewer than $r$ poles.
As rational functions with higher order poles are nowhere dense in $\set R_r^+(\C)$~\cite[p.~2739]{DGA10}
and $\widehat H_r^\ell \to \widehat H_r$
there is some $N>0$ for which all $\widehat{H}_r^\ell$ with $\ell>N$
has no higher order poles and thus
\begin{equation}
	\widehat{H}_r^\ell(z) = \sum_{j=1}^r \rho_j^{\, \ell} (z - \lambda_j^\ell)^{-1}
	\quad \forall \ell > N_1.
\end{equation}
As the pole-residue form is unique up to permutations, 
then as $\widehat H_r^\ell \to \widehat H_r$,
${\lambda}_j^{\ell} \to {\lambda}_j$
using an appropriate ordering of ${\lambda}_j^\ell$.
\end{proof}

\begin{proof}[Proof of \cref{thm:conv}]
We will show that the limit $\widehat H_r$
satisfies the first order optimality conditions~\cref{eq:opt_orig_Q}:
namely $\|Q(\ve \lambda)[H - \widehat H_r]\|_\Htwo = 0$
where $\ve\lambda$ denotes the poles of $\widehat H_r$.

Beginning by inserting an additive identity $0 = \widehat H_r^\ell - \widehat H_r^\ell$
and applying the triangle inequality:
\begin{align}
	\label{eq:conv_split1}
	\|Q(\hve \lambda)[H - \widehat H_r]\|_\Htwo
	& \le 
	\|Q(\hve \lambda)[H - \widehat H_r^\ell]\|_\Htwo
	+ \|Q(\hve \lambda)[\widehat H_r - \widehat H_r^\ell]\|_\Htwo.
\end{align}
Inserting a multiplicative identity 
$I = P(\ve\mu^\ell) + [I - P(\ve\mu^\ell)]$
into the first term and again applying the triangle inequality,
\begin{equation}\label{eq:conv_split2}
	\|Q(\ve \lambda)[H - \widehat H_r^\ell]\|_\Htwo 
	\le 
	\|Q(\ve \lambda)P(\ve\mu^\ell)[H - \widehat H_r^\ell]\|_\Htwo \\
	+
	\|Q(\ve \lambda)[I - P(\ve\mu^\ell)][H - \widehat H_r^\ell]\|_\Htwo
\end{equation}
Once again we insert a multiplicative identity, 
this time $I = Q(\ve \lambda^\ell) + [I - Q(\ve \lambda^\ell)]$
where $\ve\lambda^\ell$ denotes the poles of $\widehat H_r^\ell$
into the first term above yielding
\begin{multline}\label{eq:conv_split3}
	\|Q(\ve \lambda)P(\ve\mu^\ell)[H - \widehat H_r^\ell]\|_\Htwo
	\le 
	\|Q(\ve \lambda)Q(\ve \lambda^\ell) P(\ve\mu^\ell)[H - \widehat H_r^\ell]\|_\Htwo  \\
	+
	\|Q(\ve\lambda)[ I - Q (\ve\lambda^\ell)] P(\ve\mu^\ell)[H - \widehat H_r^\ell]\|_\Htwo. 
\end{multline}
Since $\widehat H_r^\ell$ satisfies the first order optimality conditions,
$Q(\ve \lambda^\ell)P(\ve\mu^\ell)[H - \widehat H_r^\ell] = 0$ and the first term vanishes.
Combining the bounds in \cref{eq:conv_split1}, \cref{eq:conv_split2}, and \cref{eq:conv_split3}
\begin{multline}
	\|Q(\ve \lambda)[H - \widehat H_r]\|_\Htwo
	\le  
		\|Q(\ve\lambda)[ I - Q (\ve\lambda^\ell)] P(\ve\mu^\ell)[H - \widehat H_r^\ell]\|_\Htwo
		+ \\
		\|Q(\ve \lambda)[I - P(\ve\mu^\ell)][H - \widehat H_r^\ell]\|_\Htwo
		+
		\|Q(\ve \lambda)[\widehat H_r - \widehat H_r^\ell]\|_\Htwo.
\end{multline}
Then as the induced $\Htwo$ norm is submultiplicative and $\|Q(\ve \lambda) \|_\Htwo=1$,
\begin{multline}\label{eq:conv:split}
	\|Q(\ve \lambda)[H - \widehat H_r]\|_\Htwo
	 \le 
	\|Q(\ve\lambda)[ I - Q (\ve\lambda^\ell)]\|_\Htwo
	\|H - \widehat H_r^\ell\|_\Htwo 
	\\
	+\|Q(\ve \lambda)[I - P(\ve\mu^\ell)]\|_\Htwo
	\|H - \widehat H_r^\ell\|_\Htwo
	+
	\|\widehat H_r -\widehat H_r^\ell\|_\Htwo.
\end{multline}

We now show the each of the terms terms of~\cref{eq:conv:split} vanish as $\ell \to \infty$.
The first term is multiplied by the subspace angle between the tangent space
of $\widehat H_r$ and the current iterate $\widehat H_r^\ell$:
\begin{equation}
	\sin \phi_{\max}(\set T(\ve\lambda), \set T(\ve \lambda^\ell)) = \|Q(\ve \lambda)[I - Q(\ve \lambda^\ell)]\|_\Htwo.
\end{equation}
By \cref{lem:poles} we know $\ve \lambda^\ell \to \ve \lambda$
and hence for any $\epsilon>0$ there is an $N_1$ such that
$| \lambda_j - \lambda_j^\ell|\le \epsilon$
for all $\ell > N_1$.
Then using \cref{lem:Qangle_full} we can show
\begin{equation}\label{eq:conv:split1}
	\|Q(\ve \lambda)[I - Q(\ve \lambda^\ell)]\|_\Htwo
	\le 
	C_1 \epsilon
	\quad 
	\forall \ell > N_1.
\end{equation}

The second term in~\cref{eq:conv:split} is
multiplied by the subspace angle between
the projection subspace and the tangent subspace
\begin{equation}
	\sin\phi_{\max} (\set T(\ve\lambda), \set V(\ve\mu^\ell)) =
	\|Q(\ve \lambda)[I - P(\ve\mu^\ell)]\|_\Htwo.
\end{equation}
By assumption $\lbrace \ve\mu_\star^\ell\rbrace_{\ell=0}^\infty$
has limit points $-\ove \lambda$.
Thus for any $\epsilon > 0$, there exists $N_2$ such that 
$|\conj \lambda_j^\ell + \mu_{j,t}| \le  \epsilon$ for $t=1,2,3$
where $\mu_{j,t}$ is in $\ve\mu^\ell$.
Then invoking \cref{thm:angle}, there is a constant $C_2 > 0$ such that
\begin{equation}\label{eq:conv:split2}
	\|Q(\ve\lambda) [I - P(\ve\mu^\ell)]\|_\Htwo \le C_2 \epsilon
	\quad \forall \ell > N_2.
\end{equation}

The third term in~\cref{eq:conv:split} vanishes
as by assumption $\widehat H_r^\ell \to \widehat H_r$.
Hence there exists $N_3 > 0$ such that
\begin{equation}\label{eq:conv:split3}
	\|\widehat H_r^\ell - \widehat H_r \|_\Htwo \le \epsilon
	\quad \forall \ell > N_3.
\end{equation}

Finally, we bound~\cref{eq:conv:split}.
Using \cref{eq:conv:split1}, \cref{eq:conv:split2}, and \cref{eq:conv:split3}
in \cref{eq:conv:split}
\begin{equation}
	\|Q(\ve \lambda)[H - \widehat H_r]\|_\Htwo
	 \le 
	\|H - \widehat H_r^\ell\|_\Htwo C_1 \epsilon + 
	\|H - \widehat H_r^\ell\|_\Htwo C_2 \epsilon +
	\epsilon
	\quad
	 \forall \ell > \max \lbrace N_1, N_2, N_3 \rbrace.
\end{equation}
Further as $\widehat H_r^\ell \to \widehat H_r$, 
there is some $C_4$ such that 
\begin{equation}
	\|H - \widehat H_r^\ell\| \le C_4 \quad 
	\forall \ell > N_4 
\end{equation}
Finally, combining these results
\begin{equation}
	\| Q(\ve\lambda)[H - \widehat H_r]\|_\Htwo
	\le [(C_1+C_2)C_4 + 1] \epsilon.
\end{equation}
As the choice of $\epsilon$ was arbitrary,
$\|Q(\ve\lambda)[ H - \widehat H_r]\|_\Htwo = 0$.
\end{proof}

%% file: acknowledgements.tex
\section*{Acknowledgements}
The authors would like to thank 
Christopher Beattie,
Zlatko Drma\v{c},
Mark Embree, 
Serkan Gugercin,
Christian Himpe,
Petar Mlinari\'c, 
Neeraj Sarna,
and the anonymous reviewers 
for their feedback during the preparation of this manuscript.
The first author would like to also thank the Einstein Foundation Berlin
and Paul Constantine for their support.